\documentclass[11pt]{amsart}
\usepackage{amssymb,latexsym,amsmath,graphicx,graphics,epic,eepic}

\theoremstyle{plain}
\newtheorem{theorem}{Theorem}
\numberwithin{theorem}{section}
\newtheorem{lemma}[theorem]{Lemma}
\newtheorem{proposition}[theorem]{Proposition}

\theoremstyle{definition}

\newtheorem{example}[theorem]{Example}

\newtheorem{question}[theorem]{Question}

\newtheorem{remark}[theorem]{Remark}

\newcommand{\C}{{\mathbb C}}
\newcommand{\R}{{\mathbb R}}
\newcommand{\Z}{{\mathbb Z}}
\newcommand{\Q}{{\mathbb Q}}
\renewcommand{\P}{{\mathbb P}}
\newcommand{\s}{{\mathbb S}}

              \newcommand{\M}{{\mathcal M}}

\begin{document}
\title{$G$-minimality and invariant negative spheres in $G$-Hirzebruch surfaces}
\author{Weimin Chen}
\thanks{The author is partially supported by NSF grant DMS-1065784}
\date{\today}
\maketitle

\begin{abstract}
In this paper a study of $G$-minimality, i.e., minimality of four-manifolds equipped with an
action of a finite group $G$, is initiated. We focus on cyclic actions on $\C\P^2\# \overline{\C\P^2}$, 
and our work shows that even in this simple setting, the comparison of $G$-minimality in the 
various categories, i.e., locally linear, smooth, and symplectic, is already delicate and interesting. 
For example, we show that if a symplectic $\Z_n$-action on $\C\P^2\# \overline{\C\P^2}$ has an invariant locally linear topological $(-1)$-sphere, then it must admit an invariant symplectic $(-1)$-sphere, provided that $n=2$ or $n$ is odd. For the case where $n>2$ and even, the same conclusion holds under a stronger assumption, i.e., the invariant $(-1)$-sphere is smoothly embedded. Along the way of these proofs we develop certain techniques for producing embedded invariant $J$-holomorphic two-spheres of self-intersection $-r$ under a weaker assumption of an invariant smooth $(-r)$-sphere for $r$ relatively small compared with the group order $n$. We then
apply the techniques to give a classification of $G$-Hirzebruch surfaces (i.e., Hirzebruch
surfaces equipped with a homologically trivial, holomorphic $G=\Z_n$-action) up to orientation-preserving equivariant diffeomorphisms. The main issue of the classification is to distinguish 
non-diffeomorphic $G$-Hirzebruch surfaces which have the same fixed-point set structure.
An interesting discovery is that these non-diffeomorphic $G$-Hirzebruch surfaces
have distinct equivariant Gromov-Taubes invariant, giving the first examples of such kind. 
Going back to the original question of  $G$-minimality, we show that for $G=\Z_n$, a minimal rational $G$-surface is minimal as a symplectic $G$-manifold if and only if it is minimal as a smooth $G$-manifold. The paper also contains some discussions about existence of an invariant piecewise 
linear $(-1)$-sphere and equivariant decompositions along a homology three-sphere for cyclic actions on $\C\P^2\# \overline{\C\P^2}$.
\end{abstract}

\section{Introduction}

Minimality is a basic concept in four-manifold topology. An oriented smooth four-manifold is called
minimal if it does not contain any smoothly embedded two-spheres of self-intersection $-1$; such a
two-sphere is called a $(-1)$-sphere. If a four-manifold $X$ contains a $(-1)$-sphere, then $X$ is 
naturally diffeomorphic to a connected sum $X^\prime \# \overline{\C\P^2}$, where $X^\prime$ is 
called the blowdown of $X$ along the $(-1)$-sphere. One can simplify a non-minimal four-manifold 
through a sequence of blowdowns until one gets a minimal four-manifold, and it is a fundamental 
question to understand how a four-manifold and its blowdowns are related, e.g., how their gauge 
theoretic invariants are related (cf. \cite{FS1, FS2}). When the four-manifolds and the $(-1)$-spheres 
in question are complex analytic or symplectic, the blowdown operation can be carried out in the corresponding category, giving naturally the notion of minimality in the complex analytic or the 
symplectic category (cf. \cite{McS1}).
It is well-known that there are complex surfaces which are minimal in the complex analytic 
category but not minimal as smooth four-manifolds, i.e., the Hirzebruch surfaces $F_r$ where 
$r$ is odd and $r>1$ ($F_r$ is diffeomorphic to $\C\P^2 \# \overline{\C\P^2}$ when $r$ is odd). 
On the other hand, thanks to the deep work of Taubes on the equivalence of Seiberg-Witten and 
Gromov invariants of symplectic four-manifolds, the notions of minimality in the symplectic and smooth categories are equivalent (cf. \cite{Taubes, LL0, Li}), and this has many important 
consequences in four-manifold topology.

The notion of minimality can be extended naturally to the equivariant setting. An algebraic 
surface together with a finite automorphism group $G$ is called a minimal $G$-surface if it can not
be blown down equivariantly. (Minimal rational $G$-surfaces 
played a fundamental role in the modern approach to the classical problem of classifying 
finite subgroups of the plane Cremona group, i.e., the group of birational transformations 
of the projective plane, cf. \cite{DI}.) One can similarly define the notion 
of symplectic minimality or smooth minimality in the equivariant setting. To be more concrete, a 
symplectic four-manifold with a symplectic $G$-action (resp. an oriented smooth four-manifold 
with an orientation-preserving smooth $G$-action) is called minimal if there does not exist any 
$G$-invariant set of disjoint union of symplectic (resp. smooth) $(-1)$-spheres; clearly one can 
blow down the $G$-manifold equivariantly if such a $G$-invariant set of $(-1)$-spheres exists. 
It is known that there are minimal $G$-Hirzebruch surfaces which 
are not minimal as smooth $G$-manifolds (cf. \cite{Wil2}), so whether the notion of $G$-minimality 
is the same in the symplectic and smooth categories becomes a natural question. (We shall restrict ourselves to the situation where the $G$-invariant set of
$(-1)$-spheres can be consistently oriented such that the corresponding homology classes are preserved under the $G$-action; this more restrictive assumption is automatically satisfied in the symplectic or complex analytic category.)

\vspace{1.5mm}

First of all, for a smooth four-manifold with $b_2^{+}\geq 1$ which is not diffeomorphic 
to a rational surface, the notions of minimality in all three categories (i.e., complex analytic, symplectic, or smooth) are equivalent. This continues to be true in the equivariant setting provided that the underlying manifold is also not ruled. More precisely,

\vspace{1.5mm}

\noindent{\bf Theorem 1.0.} \hspace{1mm} 
{\it Let $X$ be a complex surface with $b_2^{+}\geq 1$ or a symplectic four-manifold, which is equipped with a holomorphic or symplectic $G$-action for a finite group $G$. Suppose $X$ is 
neither rational nor ruled. Then $X$ is minimal as a complex $G$-surface or a symplectic $G$-manifold if and only if the underlying manifold $X$ is minimal as a smooth four-manifold. 
}
\vspace{1.5mm}

We remark that Theorem 1.0 is not true without the assumption that $X$ is not rational nor ruled.
In fact, there are minimal rational $G$-surfaces whose underlying smooth four-manifolds 
are not minimal (cf. \cite{DI}). For counter-examples which are non-rational ruled, consider surfaces 
which are a double branched cover of $\C\P^1\times \Sigma$ whose branch locus is a connected, nonsingular bisection of the ruled surface (here $\Sigma$ has genus $\geq 1$). With respect to the natural automorphism of order $2$, such surfaces are minimal $G$-surfaces (with $G=\Z_2$). However, the underlying surfaces are not minimal. A proof of Theorem 1.0 is given in Section 6.

With the preceding understood, as far as the comparison of $G$-minimality is concerned, the only interesting case is when the underlying four-manifold is rational or ruled. While finite automorphism groups of rational surfaces have been studied extensively in connection with the plane Cremona group (cf. \cite{DI}), general symplectic finite group actions on a rational or ruled four-manifold 
remain largely unexplored except for the case of $\C\P^2$ (cf. \cite{C3,C,C2,C1}); see also \cite{CLW}. 

In this paper, we shall take an initial step by focusing on the case of symplectic $G$-actions on 
$\C\P^2\# \overline{\C\P^2}$ where $G=\Z_n$ is a cyclic group of order $n$. Our work shows 
that even in the simple setting of $\C\P^2\# \overline{\C\P^2}$, the comparison of $G$-minimality in the various categories is already quite delicate and interesting, which, technically speaking, 
amounts to showing the existence of embedded, $G$-invariant $J$-holomorphic two-spheres under a weaker assumption of a $G$-invariant two-sphere. Moreover, we are able to apply the techniques we developed, in the guise of equivariant Gromov-Taubes invariants, to complete the smooth classification of $G$-Hirzebruch surfaces, where one encounters the problem of distinguishing non-diffeomorphic $G$-manifolds which have isomorphic fixed-point set data (cf. \cite{Wil2}).
We should point out that these examples are the first ones of non-free smooth finite group 
actions on a four-manifold whose equivariant diffeomorphism types are distinguished by
an equivariant gauge theory type invariant. The traditional technique has been through the 
knotting of the $2$-dimensional fixed-point set (cf. \cite{Gif, Go, HLM, FSS, KR}); however, 
the $G$-Hirzebruch surfaces distinguished by our method are not distinguishable by the 
traditional method (see Remark 5.4 for more details). In this sense our work has also broken 
new ground for this type of problems. We plan to explore further in this direction on a future occasion. 
(The technique for distinguishing free actions, which amounts to distinguishing the corresponding quotient manifolds, depends on whether the actions are orientation reversing or preserving.
In the former case, it is non-gauge theoretic, see \cite{CS, FS0, Gompf1, Ak, Gompf2}, and in the
latter case, it is gauge theoretic, see \cite{Ue}.)

Before stating our theorems, we remark that due to Freedman's topological surgery theory
(cf. \cite{Fr}), the question of minimality is trivial in the topological category (at least for simply 
connected four-manifolds). However, this is no longer the case for $G$-minimality. This said, 
we shall also consider $G$-minimality in the topological category. More concretely, let $X$ be an
oriented topological $4$-manifold equipped with an orientation-preserving, locally linear 
$G$-action. A $G$-invariant, topologically embedded surface $\Sigma\subset X$ is called locally linear if for any $z\in\Sigma$, there is a $G_z$-invariant neighborhood $U_z$ of $z$ in $X$ such 
that $(U_z,U_z\cap\Sigma)$ is equivariantly homeomorphic to $(\R^4,\R^2)$ with a linear 
$G_z$-action (cf. Lashof-Rothenberg \cite{LR}, p. 227); in particular, $\Sigma$ is locally flat.
It follows easily from Freedman-Quinn (cf. \cite{FQ}, Theorem 9.3A, p. 137) that a $G$-invariant 
locally linear surface has a $G$-equivariant normal bundle and hence a $G$-invariant regular
neighborhood given by the corresponding disc bundle. With this understood, if $X$ contains 
a $G$-invariant set of disjoint union of locally linear $(-1)$-spheres, then one can blow down 
$X$ equivariantly in the category of locally linear topological $G$-manifolds. We say $X$ is 
minimal as a topological $G$-manifold if no such a $G$-invariant set of disjoint union of 
locally linear $(-1)$-spheres exists in $X$. 

\begin{theorem}
Let $G=\Z_n$ be a finite cyclic group of order $n$ where either $n=2$ or $n$ is odd.
Suppose a smooth $G$-action on $X=\C\P^2\# \overline{\C\P^2}$ admits a $G$-invariant, 
locally linear $(-1)$-sphere. Let $\omega$ be any given $G$-invariant symplectic form. Then for any 
generic $G$-invariant $\omega$-compatible almost complex structure $J$, there is a $G$-invariant, 
$J$-holomorphic $(-1)$-sphere in $X$. In particular, $X$ contains a $G$-invariant, $\omega$-symplectic $(-1)$-sphere. 
\end{theorem}

Since a symplectic $\Z_n$-action on $\C\P^2$ is equivariantly diffeomorphic to a linear action 
(cf. \cite{C, C3}), Theorem 1.1 has the following corollary, where the case of pseudo-free 
holomorphic actions can be also deduced from Theorem 4.14 in \cite{Wil2}.

\vspace{1.5mm}

{\it A symplectic $\Z_n$-action on $\C\P^2\# \overline{\C\P^2}$, for $n=2$ or odd, is 
equivariantly diffeomorphic to an equivariant connected sum of a pair of linear actions on $\C\P^2$ 
and $\overline{\C\P^2}$ if and only if it admits an invariant, locally linear $(-1)$-sphere.
}
\vspace{1.5mm}

We remark that the existence of a $G$-invariant $J$-holomorphic $(-1)$-sphere asserted 
in Theorem 1.1 is quite a subtle issue. It is not true for a non-generic $J$, as there are 
examples of minimal $G$-Hirzebruch surfaces which are not minimal as smooth $G$-manifolds. 
Moreover, since the existence of a $G$-invariant locally linear $(-1)$-sphere imposes certain constraints on the representations of $G$ on the tangent spaces of the fixed points, which are not satisfied by a general symplectic (even holomorphic) $\Z_n$-action on $\C\P^2\# \overline{\C\P^2}$, Theorem 1.1 is not expected to be true for an arbitrary symplectic $\Z_n$-action. Finally, there are 
pairs of pseudo-free $G$-Hirzebruch surfaces with isomorphic local representations at the fixed 
points, such that exactly one of them contains a $G$-invariant locally linear $(-1)$-sphere (cf. \cite{Wil2}). This shows that in Theorem 1.1, one can not replace the assumption of existence of 
a $G$-invariant locally linear $(-1)$-sphere by any condition merely on the local representations at the fixed points (see Example 3.1 for more details). We remark that it is a consequence of the topological classification theorems of pseudo-free locally linear cyclic actions in \cite{Wil1, Wil2} that only one of the pseudo-free $G$-Hirzebruch surfaces in each such pair contains a $G$-invariant locally linear $(-1)$-sphere. Our approach in this paper offered an alternative proof of this fact, see Lemma 3.3 and Remark 3.4. 

For $\Z_n$-actions where $n>2$ and even, we need to impose a stronger 
assumption, i.e., the $G$-invariant $(-1)$-sphere is smoothly embedded. 
It would be interesting to know whether the smoothness condition is also necessary.

\begin{theorem}
Let $G=\Z_n$ be a cyclic group of order $n$. Suppose a smooth $G$-action on 
$X=\C\P^2\# \overline{\C\P^2}$ admits a $G$-invariant smooth
$(-1)$-sphere. Then for any $G$-invariant symplectic form $\omega$, there is a $G$-invariant,
$\omega$-symplectic $(-1)$-sphere in $X$. 
\end{theorem}

As in the case of Theorem 1.1, one similarly has the following corollary.

\vspace{1.5mm}

{\it A symplectic $\Z_n$-action on $\C\P^2\# \overline{\C\P^2}$ is equivariantly 
diffeomorphic to an equivariant connected sum of a pair of linear actions on $\C\P^2$ and 
$\overline{\C\P^2}$ if and only if it admits an invariant, smooth $(-1)$-sphere.
}
\vspace{1.5mm}

A natural question is whether a symplectic $\Z_n$-action on $\C\P^2\# \overline{\C\P^2}$ 
always possesses an invariant, topologically embedded $(-1)$-sphere which is not necessarily 
locally linear; note that without the locally-linear condition there are no additional constraints 
which must be satisfied by the local representations of the $\Z_n$-action. A particular interesting case is that of a piecewise linear $(-1)$-sphere, as such a $(-1)$-sphere has a regular 
neighborhood whose boundary is a smoothly embedded integral homology three-sphere. 

\begin{question}
Consider an arbitrary symplectic $\Z_n$-action on $X=\C\P^2\# \overline{\C\P^2}$.
\begin{itemize}
\item [{(i)}] Is there always an invariant piecewise linear $(-1)$-sphere in $X$?
\item [{(ii)}] Is there an equivariant decomposition of $X$ into $X_{+}$, $X_{-}$ along a smoothly
embedded integral homology three-sphere $\Sigma^3$ such that $X_{+}$, $X_{-}$ are an integral homology $\C\P^2\setminus B^4$, 
$\overline{\C\P^2}\setminus B^4$ respectively?
\item [{(iii)}] Is there an equivariant decomposition of $X$ as described in (ii) such that 
$\Sigma^3$ is of contact type with respect to the $\Z_n$-invariant symplectic form on $X$?
\end{itemize}
\end{question}

Note that in (ii), (iii) of Question 1.3,  if we replace the word ``integral" by ``rational", the answers 
to both questions are affirmative, cf. Lemma 2.3(2). Moreover, regarding (ii) of Question 1.3 in 
the non-equivariant setting, we should mention the work of Freedman-Taylor \cite{FT} on splitting 
smooth, simply connected four-manifolds along integral homology three-spheres.

\vspace{1.5mm}

Back to Theorems 1.1 and 1.2. What we have obtained therein can be paraphrased in terms of
the non-vanishing of certain ``equivariant Gromov-Taubes invariant", i.e., the invariant defined by 
counting $G$-invariant $J$-holomorphic $(-1)$-spheres for a $G$-invariant $J$ which may be required to satisfy certain genericity conditions. In particular, what we proved in Theorems 1.1 and 1.2 asserts that the corresponding equivariant Gromov-Taubes invariant is non-zero as long as there is a $G$-invariant locally linear or smooth $(-1)$-sphere, which is a property of the $G$-action that depends only on the equivariant homeomorphism or diffeomorphism type of the $G$-action. With this 
understood, there are potentially non-diffeomorphic $G$-Hirzebruch surfaces which have the same 
fixed-point set structure and are non-distinguishable by traditional methods.
Thus this type of results may be used to tell them apart if we could also prove a corresponding 
vanishing theorem for such equivariant Gromov-Taubes invariants.

Let $G=\Z_n$ for a fixed integer $n$, we shall classify Hirzebruch surfaces $F_r$, which is 
equipped with a homologically
trivial, holomorphic $G$-action, up to orientation-preserving equivariant diffeomorphisms. 
We will denote such a $G$-Hirzebruch surface by $F_r(a,b)$, where $(a,b)$ is an ordered 
pair of integers mod $n$ which completely determines the holomorphic $G$-action
(cf. \cite{Wil2}, \S 4, for the precise definition of $F_r(a,b)$).
Given any two $G$-Hirzebruch surfaces $F_r(a,b)$ and $F_{r^\prime}(a^\prime,b^\prime)$,
there are six types of canonical equivariant diffeomorphisms between them 
(all orientation-preserving) if certain numerical conditions are satisfied by the triples $(a,b,r)$ 
and $(a^\prime,b^\prime,r^\prime)$. (A detailed description of these canonical equivariant 
diffeomorphisms can be found at the beginning of Section 5.) Call the composition of a finite
sequence of canonical equivariant diffeomorphisms a {\it standard} equivariant
diffeomorphism. Then there is a complete set of numerical conditions for the triples 
$(a,b,r)$ and $(a^\prime,b^\prime,r^\prime)$, such that there is a standard equivariant
diffeomorphism between $F_r(a,b)$ and $F_{r^\prime}(a^\prime,b^\prime)$ if and only if
one of the numerical conditions is satisfied (cf. \cite{Wil2}, \S 4). 

On the other hand, if $F_r(a,b)$ and $F_{r^\prime}(a^\prime,b^\prime)$ are 
orientation-preservingly equivariantly diffeomorphic, then they must have isomorphic 
fixed-point set structures, which can be stated equivalently as one of a set of numerical
conditions is satisfied by the triples $(a,b,r)$ and $(a^\prime,b^\prime,r^\prime)$. However, 
this set of numerical conditions is strictly weaker than the set of conditions which guarantees 
a standard equivariant diffeomorphism between $F_r(a,b)$ and $F_{r^\prime}(a^\prime,b^\prime)$.
With this understood, the main task of our classification is to show, using the technique of
equivariant Gromov-Taubes invariants, that one of the stronger set of conditions must be 
satisfied.


With the preceding understood, we shall formulate our classification as follows. 

\begin{theorem}
Two $G$-Hirzebruch surfaces are orientation-preservingly equivariantly diffeomorphic
iff there is a standard equivariant diffeomorphism between them. 
\end{theorem} 

We shall point out that an analogous classification for pseudo-free $G$-Hirzebruch 
surfaces was obtained by Wilczynski (cf. \cite{Wil2}, Theorem 4.2) as an application 
of the topological classification of pseudo-free locally linear cyclic actions on simply 
connected four-manifolds (cf. \cite{Wil1,Wil2}). In fact it was actually shown that two 
pseudo-free $G$-Hirzebruch surfaces are equivarianly diffeomorphic if and only if they are equivariantly homeomorphic (cf. \cite{Wil2}, Theorem 4.14). It is not known whether this 
coincidence of topological and smooth classifications of $G$-Hirzebruch surfaces continues to 
hold in the non-pseudo-free case. However, we shall point out that a negative answer would 
imply that the smoothness condition in Theorem 1.2 is necessary.

\vspace{1.5mm}

Now we discuss some of the technical aspects of this paper. It is well-known that in the $J$-holomorphic curve
theory in dimension four, the presence of $J$-holomorphic curves of negative self-intersection causes
considerable complications in the analysis of singularity or intersection patterns of $J$-holomorphic 
curves. One basic approach to get around this issue is to work with generic almost complex structures.
The basic fact is that for a generic $J$, the only $J$-holomorphic curves of negative self-intersection 
are $(-1)$-spheres. Therefore, by working with the corresponding minimal symplectic four-manifolds
and by working with generic almost complex structures, one can avoid the issue of $J$-holomorphic 
curves of negative self-intersection. With this said, however, in various applications of $J$-holomorphic
curves one is often forced to work with non-generic almost complex structures. See Li-Zhang 
\cite{LZ} and McDuff-Opshtein \cite{McDO} for the recent articles on this topic. 

For $J$-holomorphic curves in the equivariant setting (or more generally the orbifold setting), one has to work with $G$-invariant almost complex structures. Even though one can choose generic 
$G$-invariant $J$, these almost complex structures are not generic in the usual sense; in particular,
one has to face the presence of $J$-holomorphic curves of negative self-intersection. Moreover, in the case of rational or ruled, additional complication is caused by the fact that even though the group 
action may be minimal, the underlying manifold may not be (see the remarks following Theorem 1.0.). For these reasons, the general theory of equivariant Gromov-Taubes invariants is still under construction. Finally, we remark that in the equivariant setting, information about local representations at the fixed points of the $G$-action is extremely important in analyzing singularity or intersection patterns of $G$-invariant $J$-holomorphic curves.

\vspace{1.5mm}

With the preceding understood, two technical results of this paper concerning $J$-holomorphic curves in $X=\C\P^2\#\overline{\C\P^2}$ or $\s^2\times \s^2$ for an {\it arbitrary} $J$ are worth mentioning. To describe the result for $X=\C\P^2\#\overline{\C\P^2}$, let $\omega$ be any given symplectic form
on $X$ and write the canonical class $c_1(K_\omega)=-3e_0+e_1$ for some basis $e_0,e_1\in
H^2(X)$ such that $e_0^2=-e^2_1=1$ and $e_0\cdot e_1=0$. Then one has the following alternative:

\vspace{1.5mm}

{\it For any $\omega$-compatible $J$, either $e_1$ is represented by an embedded 
$J$-holomorphic two-sphere, or $X$ admits a fibration by embedded $J$-holomorphic two-spheres 
in the class $e_0-e_1$, together with a $J$-holomorphic section $C_0$ of self-intersection $C_0^2<-1$.
}

\vspace{1.5mm} 

This result is the content of Lemma 2.3; there is a corresponding result for $X=\s^2\times \s^2$
given in Lemma 2.4. Section 2 is devoted to the proofs of these lemmas. 

With Lemma 2.3 in hand and assuming $J$ is $G$-invariant,
the main task in the proof of Theorems 1.1 and 1.2 is to show that, with 
the existence of a $G$-invariant topological or smooth $(-1)$-sphere, one can force the 
$J$-holomorphic section $C_0$ from Lemma 2.3 to have self-intersection $-1$ when $J$ is chosen
to be a certain generic $G$-invariant almost complex structure. Sections 3 and 4 are occupied by these
discussions, with Section 3 devoted to the case of pseudo-free actions and Section 4 to 
non-pseudo-free actions which requires a different approach. 

In Section 5 we extend the techniques developed in Sections 3 and 4, and show that with the
existence of a $G$-invariant smooth $(-r)$-sphere where $r\geq 0$ and relatively small, one can 
force the $J$-holomorphic section $C_0$ from Lemma 2.3 or 2.4 to have self-intersection $-r$ 
when $J$ is chosen to be a certain generic $G$-invariant almost complex structure. With this 
understood, the main task in proving Theorem 1.4 is to use the corresponding equivariant 
Gromov-Taubes invariant to distinguish $G$-Hirzebruch surfaces $F_r(a,b)$ and $F_{r+n}(a,b)$
(which have the same fixed-point set structure), showing that for exactly one of the 
$G$-Hirzebruch surfaces, the corresponding equivariant Gromov-Taubes invariant is non-vanishing.
This is the content of Proposition 5.3. 

Finally, in Section 6 we return to the original question about symplectic and smooth 
minimality of symplectic $G$-manifolds. We consider the question of minimality of minimal rational $G$-surfaces in the category of symplectic (resp. smooth or even locally linear topological)
$G$-manifolds. Using some general facts about minimal rational $G$-surfaces, we show that 
such a $G$-surface must be minimal as a topological $G$-manifold, unless it is a conic bundle with singular fibers or a Hirzebruch surface. Furthermore, specializing to the case of $G=\Z_n$, we show that a minimal $G$-conic bundle 
with singular fibers must be minimal as a topological $G$-manifold (cf. Proposition 6.2). Combining this result with Theorem 1.2, we obtain the following theorem.

\begin{theorem}
Let $G=\Z_n$ be a cyclic group of order $n$. Then a minimal rational $G$-surface is minimal as a symplectic $G$-manifold if and only if it is minimal as a smooth $G$-manifold. 
\end{theorem}

Throughout this paper, we shall not distinguish the notations for a $J$-holomorphic curve, its homology
class, or the cohomology class Poincar\'{e} dual to it, as long as there will be no confusion caused by
the ambiguity. 

\vspace{2mm}

\noindent{\bf Acknowledgments:}
Part of this paper was written during a visit to Capital Normal University, Beijing. I wish to thank 
Professor Fuquan Fang for the invitation and very warm hospitality. I am also in debt to Slawomir
Kwasik for enlightening conversations and various helps, and to Tian-Jun Li for bringing his joint 
work with Weiyi Zhang to my attention. Finally, I wish to thank the referee for several suggestions which improved the exposition of the paper. 

\section{Preliminary lemmas}

We first consider the case where $X=\C\P^2\# \overline{\C\P^2}$. Fix a basis $e_0,e_1$ 
of $H^2(X)$ such that 
$$
e_0^2=1, e_1^2=-1, \mbox{ and } e_0\cdot e_1=0.
$$
The following lemma shows that such a basis is unique up to a sign change.

\begin{lemma}
Suppose $f_0,f_1\in H^2(X)$ such that $f_0^2=-f_1^2=1$. Then 
$$
f_0=\pm e_0, f_1=\pm e_1.
$$
\end{lemma}

\begin{proof}
Write $f_0=ae_0+be_1$ for $a,b\in\Z$. Then $1=f_0^2=a^2-b^2=(a-b)(a+b)$. It follows easily
that $a=\pm 1$ and $b=0$. The claim for $f_1$ follows similarly.

\end{proof}

Let $\omega$ be any given symplectic form on $X$, and let $K_\omega$ be the canonical 
bundle. Then $c_1^2(K_\omega)=8$ implies easily that $c_1(K_\omega)=\pm 3e_0\pm e_1$. 
Without loss of generality, we assume that $c_1(K_\omega)=-3e_0 + e_1$. Note that with this 
choice, $e_1$ is represented by a $\omega$-symplectic $(-1)$-sphere (cf. \cite{LL0}, Theorem A) 
which, by Lemma 2.1, is the only such class. In particular, $\omega(e_1)>0$. We shall also 
consider the ``fiber" class $F=e_0-e_1$. We claim that $\omega(F)>0$ as well; in particular, 
$\omega(e_0)>\omega(e_1)$. To see this, we blow down $(X,\omega)$ along the symplectic 
$(-1)$-sphere representing $e_1$, and denote the resulting symplectic four-manifold by 
$(X^\prime,\omega^\prime)$. Then $X^\prime$ is a symplectic $\C\P^2$, with 
$c_1(K_{\omega^\prime})=-3e_0$ where we naturally identify $e_0$ as a class in $X^\prime$.
Then $e_0$ can be represented by an embedded, $\omega^\prime$-symplectic two-sphere $S$ in 
$X^\prime$ passing through the center of the blowdown operation. The proper transform of $S$ in 
$X$ is an embedded, $\omega$-symplectic two-sphere representing $F$, hence $\omega(F)>0$ 
as claimed.

The following lemma should be well-known to the experts. However, we did not find the exact statements
of the lemma in the literature, hence for the sake of completeness, we include it here.

\begin{lemma}
Let $SW_X$ denote the Seiberg-Witten invariant of $(X,\omega)$ defined in the Taubes chamber,
where $X=\C\P^2\# \overline{\C\P^2}$. Then 
$$
SW_X(e_1)=\pm 1, \mbox{ and } SW_X(F)=\pm 1. 
$$
\end{lemma}

\begin{proof}
Since $b_1(X)=0$, the wall-crossing number is $\pm 1$ (cf. \cite{KM}). Consequently,
$$
|SW_X(e_1)\pm SW_X(K_\omega-e_1)|= 1, \mbox{ and } |SW_X(F)\pm SW_X(K_\omega -F)|=1. 
$$
The lemma follows by showing that $SW_X(K_\omega-e_1)=SW_X(K_\omega -F)=0$. The key point is
that $\omega(K_\omega-e_1)= -3\omega(e_0)<0$ and 
$$
\omega(K_\omega-F)=-4\omega(e_0)+2\omega(e_1)<-2\omega(e_1)<0,
$$
so that if any of $SW_X(K_\omega-e_1)$, $SW_X(K_\omega -F)$ is nonzero, by Taubes' theorem
\cite{Taubes} the corresponding class is represented by pseudo-holomorphic curves, contradicting the 
above negativity of the symplectic areas. The lemma is proved. 

\end{proof}

Now let $J$ be any given $\omega$-compatible almost complex structure on $X$. Since
$SW_X(e_1)=\pm 1$, by Taubes' theorem \cite{Taubes} there is a finite set of $J$-holomorphic curves $\{C_i|i\in I\}$, such that $e_1=\sum_{i\in I} m_i C_i$, where $m_i>0$.

\begin{lemma}
One has the following alternative: (1) The set $\{C_i|i\in I\}$ consists of a single element $C_0$
which is an embedded $J$-holomorphic two-sphere, and $e_1=C_0$, or (2) 
$X=\C\P^2\# \overline{\C\P^2}$ admits a 
$\s^2$-fibration over $\s^2$ such that each fiber is an embedded $J$-holomorphic two-sphere 
in the class $F$, and furthermore, the set $\{C_i|i\in I\}=\{C_0\}\sqcup \{C_i|i\in I_0\}$ where $C_0$ and each $C_i$ are a section and a fiber of the $\s^2$-fibration respectively, and 
$e_1=C_0+ \sum_{i\in I_0} m_i C_i$.
\end{lemma}

\begin{proof}
We consider first the case where for each $i\in I$, $C_i^2<0$. If for all $i\in I$, $C_i^2\leq -2$,
then the adjunction formula would imply that $c_1(K_\omega)\cdot C_i\geq -2-C_i^2\geq 0$,
which would then give 
$$
c_1(K_\omega)\cdot e_1=\sum_{i\in I} m_ic_1(K_\omega)\cdot C_i\geq 0.
$$
But $c_1(K_{\omega})\cdot e_1=-1$, which is a contradiction. Hence there must be a 
$C_0\in \{C_i|i\in I\}$ such that $C_0^2=-1$. Lemma 2.1 implies that $C_0=e_1$, from which 
it follows easily that $\{C_i|i\in I\}$ consists of a single element $C_0$.  The adjunction formula
implies that $C_0$ is an embedded two-sphere. This belongs to case (1).

For case (2) let $I_0=\{i\in I| C_i^2\geq 0\}$ and assume that $I_0$ is nonempty. 
For each $i\in I_0$, we
write $C_i=a_ie_0-b_ie_1$ for $a_i,b_i\in\Z$, and note that $C_i^2\geq 0$ is equivalent to 
$a_i^2-b_i^2\geq 0$. On the other hand, since $SW_X(F)=\pm 1$, $F$ can be represented 
by $J$-holomorphic curves by Taubes' theorem \cite{Taubes}. By positivity of intersection 
of $J$-holomorphic curves, together with the fact that $C_i^2\geq 0$, $i\in I_0$, we see that 
$F\cdot C_i\geq 0$ is true for each $i\in I_0$, which can be translated into $a_i-b_i\geq 0$. 
It follows easily that $a_i+b_i\geq 0$ and $a_i>0$ for each $i\in I_0$. 

Now we set $\Theta=\sum_{i\in I\setminus I_0} m_iC_i$. Note that for each $i\in I_0$,
$\Theta\cdot C_i\geq 0$ by the positivity of intersection of $J$-holomorphic curves. It follows
easily that $\Theta^2<0$ as $e_1^2=-1<0$. Writing $\Theta=a_0e_0-b_0e_1$ for
some $a_0,b_0\in\Z$, we then have
$$
e_1=(a_0+\sum_{i\in I_0} m_i a_i) e_0- (b_0+\sum_{i\in I_0} m_ib_i)e_1,
$$
which gives $a_0+\sum_{i\in I_0} m_i a_i=0$ and $b_0+\sum_{i\in I_0} m_ib_i +1=0$. 
Consequently, unless $a_i-b_i=0$ for all $i\in I_0$, we must have 
$$
a_0+b_0=-\sum_{i\in I_0} m_i(a_i+b_i) -1<0 \mbox{ and } a_0-b_0
=-\sum_{i\in I_0} m_i(a_i-b_i) +1\leq 0,
$$
implying $\Theta^2=a_0^2-b_0^2\geq 0$ which is a contradiction. Hence $C_i=a_iF$ for
all $i\in I_0$. By the adjunction inequality, $C_i^2+c_1(K_\omega)\cdot C_i+2\geq 0$, 
which implies that for each $i\in I_0$, $a_i=1$ and $C_i$ is an embedded $J$-holomorphic 
two-sphere with self-intersection $0$.

To prove the rest of the assertions in case (2), we let $\M$ be the moduli space of embedded 
$J$-holomorphic two-spheres with self-intersection $0$ representing the class $F$. Then $\M$ 
is smooth (cf. \cite{McDS}, Lemma 3.3.3) and has dimension $2$, and we have just shown that 
$\M\neq \emptyset$ (in fact, $C_i\in \M$ for each $i\in I_0$). Furthermore, 
$\M$ must be compact. To see this, suppose $F=\sum_j m_j^\prime C_j^\prime$ for some 
$J$-holomorphic curves $C_j^\prime$ with multiplicity $m_j^\prime$. Then for any $i\in I_0$,
$0=F\cdot C_i=\sum_j m_j^\prime C_j^\prime\cdot C_i$. Note that $C_j^\prime\cdot C_i\geq 0$ 
for all $j$, from which it follows that for each $j$, $C_j^\prime\cdot C_i=0$ and 
$C_j^\prime$ is a positive multiple of $F$. It follows easily that $\{C_j^\prime\}$ consists of 
a single element $C_j^\prime$ with $m_j^\prime=1$. Furthermore, the adjunction formula 
implies that $C_j^\prime$ is an embedded two-sphere, hence $C_j^\prime\in \M$. By Gromov 
compactness, $\M$ is compact.

The $J$-holomorphic curves in $\M$ gives rise to a $\s^2$-fibration over $\s^2$ structure on $X$. Since $1=e_1\cdot F=(\Theta +\sum_{i\in I_0} m_i C_i)\cdot F=\Theta\cdot F$, we see immediately 
that $\{C_i|i\in I\setminus I_0\}$ consists of a single element $C_0$ with multiplicity $m_0=1$,
and $C_0\cdot F=1$. The latter implies easily that $C_0$ is a section of the $\s^2$-fibration on 
$X$. Finally, it is clear that $e_1=C_0+\sum_{i\in I_0} m_i C_i$. This finishes the proof of the
lemma.

\end{proof}

We remark that it is easily seen that $C_0^2$ and $e_1^2$ have the same parity. Consequently,
$C_0$ is an embedded $J$-holomorphic two-sphere with odd, negative self-intersection.

Lemma 2.3 has an analog in the case of $X=\s^2\times \s^2$ which will be used in the proof
of Theorem 1.4 in Section 5. The proof is similar, so we shall only sketch it here. 
Let $\omega$ be any given symplectic structure on $X$. Then there is a basis 
$e_1,e_2\in H^2(X)$ where $e_1^2=e^2_2=0$ and $e_1\cdot e_2=1$, such that the
canonical class $c_1(K_\omega)=-2e_1-2e_2$ (cf. \cite{LL}). Observe that $[\omega]^2>0$ 
implies that $\omega(e_1)$, $\omega(e_2)$ are non-zero and have the same sign. 
Together with the fact that $c_1(K_\omega)\cdot [\omega]<0$, it implies that both $\omega(e_1)$, 
$\omega(e_2)$ are positive. Finally, an argument involving wall-crossing as in Lemma 2.2 
shows that $SW_X(e_i)=\pm 1$ for $i=1,2$.

Let $J$ be any given $\omega$-compatible almost complex structure on $X$. By Taubes' theorem \cite{Taubes}, for any $j=1,2$, $e_j$ is represented by $J$-holomorphic curves. Without loss of generality, we only consider the case of $e_1$. Then there is a finite set of $J$-holomorphic 
curves $\{C_i|i\in I\}$ such that $e_1=\sum_{i\in I} m_i C_i$ for some $m_i>0$.

\begin{lemma}
One has the following alternative:  (1) The set $\{C_i|i\in I\}$ consists of a single element $C_0$
which is an embedded $J$-holomorphic two-sphere, and $e_1=C_0$, or (2) $X=\s^2\times\s^2$ admits a 
$\s^2$-fibration over $\s^2$ such that each fiber is an embedded $J$-holomorphic two-sphere 
in the class $e_2$, and furthermore, the set $\{C_i|i\in I\}=\{C_0\}\sqcup \{C_i|i\in I_0\}$ where 
$C_0$ and each $C_i$ are a section and a fiber of the $\s^2$-fibration respectively, and 
$e_1=C_0+ \sum_{i\in I_0} m_i C_i$.
\end{lemma}

\begin{proof}
We set $I_0=\{i\in I| C_i^2\geq 0\}$. Then as we argued in the previous lemma, if $I_0=\emptyset$,
i.e., $C_i^2<0$ for all $i\in I$, then in the present case as $X$ is even, $C_i^2\leq -2$ for all
$i\in I$, so that $c_1(K_\omega)\cdot C_i\geq 0$ for all $i\in I$. This would contradict 
$c_1(K_\omega)\cdot e_1=-2$, hence we must have $I_0\neq \emptyset$.

Then there are two possibilities: (1) $I_0=I$, or (2) $I_0\neq I$. In the former case, it is easily
seen that $C_i^2=0$ for all $i$, as $e_1^2=0$, and furthermore, since $e_1$ is primitive, 
the set $\{C_i|i\in I\}$ must consist of a single element $C_0$ such that $e_1=C_0$. By the
adjunction formula, $C_0$ is an embedded $J$-holomorphic two-sphere. 

In the latter case where $I\setminus I_0\neq \emptyset$, we set 
$\Theta=\sum_{i\in I\setminus I_0} m_i C_i$. Then by a similar argument as in Lemma 2.3,
we have $\Theta^2\leq 0$. Moreover, if $\Theta^2=0$, we must have $\Theta\cdot C_i=0$
and $C_i^2=0$ for all $i\in I_0$.

To proceed further, for each $i\in I_0$ we write $C_i=a_i e_1+b_i e_2$ where $a_i,b_i\in \Z$. 
Then $C_i^2\geq 0$ is equivalent to $a_ib_i\geq 0$. Now $0<\omega(C_i)=a_i\omega(e_1)+
b_i\omega(e_2)$ implies immediately that $a_i,b_i\geq 0$. 

Now we write $\Theta=a_0e_1+b_0e_2$ for some $a_0,b_0\in\Z$. Then 
$$
e_1=(a_0+\sum_{i\in I_0} m_ia_i)e_1+ (b_0+\sum_{i\in I_0} m_ib_i)e_2. 
$$
If there is an $a_i>0$, then $a_0=1-\sum_{i\in I_0} m_i a_i\leq 0$. On the other hand, 
$b_0=-\sum_{i\in I_0} m_ib_i\leq 0$, which implies that $\Theta^2=2a_0b_0\geq 0$.
Since $\Theta^2\leq 0$, we must have $\Theta^2=0$, which means either $a_0=0$ or
$b_0=0$. We claim this is a contradiction. To see it, recall that there is an $a_i>0$, and 
$\Theta\cdot C_i=0$ for all $i\in I_0$. It follows easily that $b_0=0$. On the other hand,
$0<\omega (\Theta)=a_0\omega (e_1)$, so that $a_0>0$ must be true. But this contradicts
$a_0=1-\sum_{i\in I_0} m_i a_i\leq 0$, hence our claim follows. 

This shows that $a_i=0$ for all $i\in I_0$. Furthermore, as in the proof of Lemma 2.3,
the adjunction inequality implies that $b_i=1$ for all $i\in I_0$. Hence for each $i\in I_0$,
$C_i=e_2$ and is an embedded $J$-holomorphic two-sphere of self-intersection $0$. 

Similarly, $1=e_1\cdot e_2=\Theta\cdot e_2$ implies that $\{C_i|i\in I\setminus I_0\}$
consists of a single element $C_0$ with multiplicity $m_0=1$, and $C_0\cdot e_2=1$. 
Furthermore, the existence of $C_i$, $i\in I_0$,  gives rise to a $\s^2$-fibration over $\s^2$ 
structure on $X$, where each fiber is an embedded $J$-holomorphic two-sphere 
in the class $e_2$, and $C_0$ is a section of the $\s^2$-fibration. 
Finally, we note that $e_1=C_0+\sum_{i\in I_0} m_i C_i$. This finishes off the proof.

\end{proof}

We note that it follows easily from the proof that $C_0$ is an embedded $J$-holomorphic 
two-sphere with even, negative self-intersection.

We end this section with the following remarks. Suppose $X=\C\P^2\# \overline{\C\P^2}$
is given a smooth $G$-action and $\omega$ is a $G$-invariant symplectic form on $X$. 
We first note that the $G$-action must be homologically trivial in view of Lemma 2.1 and the fact 
that $\omega(e_0)>0$ and $\omega(e_1)>0$.

Suppose the almost complex structure $J$ in Lemma 2.3 is chosen to be $G$-invariant. 
Then in case (1) $C_0$ must be $G$-invariant. This is because for any $g\in G$, 
$(g\cdot C_0)\cdot C_0=C_0^2=-1$, so that $g\cdot C_0$ and $C_0$ can not be distinct
$J$-holomorphic curves. A similar argument shows that in case (2), $C_0$ is also
$G$-invariant, and the $\s^2$-fibration on $X$ is $G$-invariant. 

Note that Theorems 1.1 and 1.2 follow immediately if Lemma 2.3(1) is true. Hence without 
loss of generality, we shall assume case (2) of the lemma in the next two sections. 

\section{Invariant $(-1)$-spheres of pseudo-free actions}

In this section and the next one, we give a proof for Theorems 1.1 and 1.2; in particular, 
$X=\C\P^2\# \overline{\C\P^2}$.  In this section, we shall consider Theorem 1.1 for  
the case where the action is pseudo-free. (Recall that a finite group action is called 
{\it pseudo-free} if points of nontrivial isotropy are isolated.) 
Our goal is to show that by picking generic $G$-invariant $J$, one could manage to force 
the $J$-holomorphic two-sphere $C_0$ from Lemma 2.3(2) to have self-intersection $-1$. 
Note that so far we have not used the assumption that $X$ contains a $G$-invariant 
locally linear $(-1)$-sphere. On the other hand, note also that for a $G$-Hirzebruch surface 
$F_r$ with $r$ odd, such a $G$-invariant fibration with a $G$-invariant section of odd, 
negative self-intersection exists automatically. In order to better understand the role played 
by the $G$-invariant locally linear $(-1)$-sphere, we begin with the following example. 

\begin{example}
Recall that for an orientation-preserving, locally linear $\Z_n$-action on an oriented four-manifold, the representation on the tangent space at a fixed point is given by a pair of integers mod $n$ after fixing a generator of $\Z_n$. The pair of weights is called the rotation numbers, which is uniquely determined 
up to a change of order or a simultaneous change of sign. When the $\Z_n$-action preserves an almost complex structure (e.g. being symplectic or holomorphic), the complex structure on the tangent spaces 
picks up a canonical sign so that the rotation numbers are uniquely determined only up to a change of
order in this case. 

With the preceding understood, let $F_r(a,b)$ be the $G$-Hirzebruch surface $F_r$ 
with a pseudo-free cyclic automorphism group $G$, such that after fixing an appropriate 
generator of $G$,  the rotation numbers at the four isolated fixed points are $(a,\pm b)$, 
$(-a, \pm (b+ra))$, with the second number in each pair standing for the weight in the fiber 
direction (cf. \cite{Wil2}, \S 4, for the precise definition). We shall consider a pair of examples: 
$F_1(1,3)$ and $F_{11}(3,1)$, with $G=\Z_7$. Clearly, $F_1(1,3)$ contains a $G$-invariant 
holomorphic $(-1)$-sphere, i.e., the zero-section $E_0$. Moreover, $F_1(1,3)$ and $F_{11}(3,1)$ 
have the same set of rotation numbers. To see this, we note that the rotation numbers of $F_1(1,3)$ 
are $(1,\pm 3)$, $(-1, \pm 4)$, while the rotation numbers of $F_{11}(3,1)$ are $(3,\pm 1)$, 
$(-3, \pm 34)=(4, \mp 1)$.  After a simultaneous change of sign on $(3,-1)$ and $(4,1)$, the rotation 
numbers of $F_{11}(3,1)$ match up exactly with the rotation numbers of $F_1(1,3)$ as unordered pairs. 

We claim that $F_{11}(3,1)$ does not contain any $G$-invariant locally linear $(-1)$-sphere (while
$F_1(1,3)$  contains a $G$-invariant holomorphic $(-1)$-sphere). The reason is that, if 
$F_{11}(3,1)$ contain a $G$-invariant locally linear $(-1)$-sphere, we can equivariantly blow down 
both $F_{11}(3,1)$ and $F_1(1,3)$ to get two locally linear, pseudo-free $\Z_7$-actions on 
$\C\P^2$ (cf. \cite{Fr}), which have the same rotation numbers. By Theorem 4.1 in \cite{Wil1},
the two $\Z_7$-actions on $\C\P^2$ are equivariantly homeomorphic. 
This then implies that $F_{11}(3,1)$ and $F_1(1,3)$ are equivariantly homeomorphic. However,
by Theorem 4.2(2) in \cite{Wil2}, $F_1(1,3)$, $F_{11}(3,1)$ are equivariantly diffeomorphic to
$F_7(1,3)$ and $F_7(3,1)$ respectively, and by Theorem 4.11(2) in \cite{Wil2}, $F_7(1,3)$ and 
$F_7(3,1)$ are not equivariantly homeomorphic, which is a contradiction. 

With the preceding understood, the assumption on the existence of a $G$-invariant locally linear
$(-1)$-sphere in Theorem 1.1 enters in the proof in such a way that it gives an alternative proof
that $F_{11}(3,1)$ does not contain any $G$-invariant locally linear $(-1)$-sphere. See Lemma 3.3
and Remark 3.4 for more details. 

\end{example}

Let $C$ be the $G$-invariant locally linear $(-1)$-sphere in $X$. Since the $G$-action is pseudo-free,
the induced action on $C$ must be effective and $C$ contains exactly two fixed points 
of the $G$-action on $X$. We shall orient $C$ so that the class of $C$ equals $e_1$, and with this 
choice of orientation on $C$, the rotation numbers of the $G$-action at the two fixed points contained 
in $C$ can be written as unordered pairs $(1, a)$ and $(-1,a+1)$ for some $a\in\Z$ mod $n$ after 
fixing an appropriate generator of $G$, where the second number in each pair stands for the weight in
the normal direction. (Note that no simultaneous change of sign is allowed here as the orientation of
$C$ is fixed.) We denote by $q_1,q_2$ the fixed points whose rotation numbers are $(1, a)$, 
$(-1,a+1)$ respectively. We shall fix the above generator of $G$ for the rest of this section and
in the next one, with all the rotation numbers or weights in reference of this generator of $G$. 
Finally, we observe that since the $G$-action is pseudo-free, the order $n$ of $G$ must be odd as both $a$ and $a+1$ are co-prime to $n$ and one of them is even.

\begin{lemma}
There is a $G$-equivariant complex line bundle $E$ over $X$ such that (i) $c_1(E)=e_1$, (ii)
the weights of the $G$-action on the fibers $E_{q_1}$, $E_{q_2}$ are $a$ and $a+1$ respectively,
and are zero at the other fixed points of the $G$-action.
\end{lemma}

\begin{proof}
We shall define $E$ in a neighborhood of $C$ first. To this end, we consider the following 
$\Z_n$-action on $\C\P^2$:

$$
\mu\cdot [z_0:z_1:z_2]= [z_0: \mu z_1:\mu^{-a} z_2], \mbox{ where } \mu=\exp(2\pi /n).
$$
Note that the local representations at the fixed points $[1:0:0], [0:1:0]$ are $(1,-a)$ and
$(-1, -a-1)$ respectively. 

Consider the complex line bundle $E$ on $\C\P^2$, where $E=\s^5\times_{\s^1} \C$,
with the $\s^1$-action on $\s^5\times \C$ given by 
$$
\lambda\cdot (z_0,z_1,z_2,t )=(\lambda z_0,\lambda z_1,\lambda z_2, \lambda^{-1} t),\;\;
\lambda\in \s^1,\;\; t\in\C.
$$
There is a specific lifting of the $\Z_n$-action on $\C\P^2$ to $E$, given on $\s^5\times \C$ by
$$
\mu\cdot (z_0,z_1,z_2,t)= (\mu^a z_0, \mu^{a+1} z_1, z_2,t). 
$$
It is easy to see that the weight of the action on the fiber at $[1:0:0]$ is $a$,
on the fiber at $[0:1:0]$ is $a+1$, and on the fiber at $[0:0:1]$ is $0$. 

Now we give $\C\P^2$ the opposite orientation and consider $E$ as a $G$-equivariant complex
line bundle over $\overline{\C\P^2}$. The $G$-equivariant section of $E$ given by 
$(z_0,z_1,z_2)\mapsto (z_0,z_1,z_2,z_2^{-1})$ has a pole at $z_2=0$ and defines a 
trivialization of $E$ in the complement of it.  If we fix an orientation-preserving identification between
a $G$-invariant regular neighborhood of $C$ in $X$ and a $G$-invariant regular
neighborhood of $z_2=0$ in 
$\overline{\C\P^2}$, $E$ defines a $G$-equivariant complex line bundle in a $G$-invariant 
neighborhood of $C$ which can be extended trivially and $G$-equivariantly over the rest of $X$.

It remains to verify that $E$ has the properties (i) and (ii). The latter is clear from the construction
of $E$. As for (i), we note that $E$ as a complex line bundle over $\C\P^2$ admits a non-vanishing section with a pole at the complex line $z_2=0$, so that $c_1(E)$ is Poincar\'{e} dual to the
negative of the class of complex lines in $\C\P^2$. After reversing the orientation of the manifold,
$c_1(E)$ is Poincar\'{e} dual to the class of complex lines, from which it follows easily that
$c_1(E)=C=e_1$. 

\end{proof}

The assumption on the existence of a $G$-invariant locally linear $(-1)$-sphere comes into play 
through the following lemma.

\begin{lemma}
For any $G$-invariant $\omega$-compatible almost complex structure $J$ on $X$, the rotation numbers
determined by the corresponding complex structure on the tangent spaces are $(1,a)$ and $(-1,a+1)$
at $q_1,q_2$ respectively, and $(1,-a)$ and $(-1,-a-1)$ at the other two fixed points. 
\end{lemma}

\begin{proof}
We first verify the assertion of the lemma for the fixed points $q_1,q_2$. We shall accomplish this by
computing the virtual dimension of the moduli space of Seiberg-Witten equations associated to $E$, 
which is denoted by $d(E)$, and show that it is integral only when the rotation numbers at $q_1,q_2$ 
are as claimed.

The formula for $d(E)$ is given in \cite{C2}, Appendix A (see also \cite{C1}, Lemma 3.3), according to
which
$$
d(E)=\frac{1}{n} (c_1(E)^2-c_1(E)\cdot c_1(K_\omega))+ I_{q_1}+I_{q_2}=I_{q_1}+I_{q_2},
$$
where $ I_{q_1},I_{q_2}$ are contributions from the fixed points $q_1,q_2$. There are no contributions 
from the other fixed points because the weights of the $G$-action on the corresponding fibers of $E$ are zero by Lemma 3.2. 

Suppose the rotation numbers at $q_1$ are $(1,a)$, then $I_{q_1}$ is given by 
$$
I_{q_1}=\frac{1}{n} \sum_{x=1}^{n-1} \frac{2(\mu^{ax}-1)}{(1-\mu^{-x})(1-\mu^{-ax})},
\mbox{ where  } \mu=\exp(2\pi/n).
$$
Similarly, if the rotation numbers at $q_2$ are $(-1,a+1)$, then
$$
I_{q_2}=\frac{1}{n} \sum_{x=1}^{n-1} \frac{2(\mu^{(a+1)x}-1)}{(1-\mu^{x})(1-\mu^{-(a+1)x})}.
$$
One can easily check that $I_{q_1}=-I_{q_2}$, and it follows that $d(E)=0$ in this case. 

Suppose the rotation numbers at $q_1$ are $(-1,-a)$ instead, then
$$
I_{q_1}=\frac{1}{n} \sum_{x=1}^{n-1} \frac{2(\mu^{ax}-1)}{(1-\mu^{x})(1-\mu^{ax})}=
-\frac{2}{n} \sum_{x=1}^{n-1} \frac{1}{(1-\mu^{x})}=-\frac{n-1}{n}.
$$
If the rotation numbers at $q_2$ are still $(-1,a+1)$, then (cf. \cite{C1}, Example 3.4)
$$
I_{q_2}=\frac{1}{n} \sum_{x=1}^{n-1} \frac{2(\mu^{(a+1)x}-1)}{(1-\mu^{x})(1-\mu^{-(a+1)x})}
=\frac{1}{n} \sum_{x=1}^{n-1} \frac{2\mu^{(a+1)x}}{1-\mu^{x}}=\frac{-(n-1)+2a}{n}, 
$$
where $a$ is the unique integer satisfying $0\leq a<n$ for the given congruence mod $n$ class.
With this, 
$$
d(E)=I_{q_1}+I_{q_2}=-\frac{n-1}{n} +\frac{-(n-1)+2a}{n}=\frac{2(a+1-n)}{n},
$$
which is non-integral because $n$ is odd and 
$a+1\neq 0 \pmod{n}$. One can similarly verify that in all other cases, i.e., when the rotation 
numbers are $(-1,-a), (1, -a-1)$, or $(1,a), (1,-a-1)$, $d(E)$ is non-integral. This proves the 
assertion for $q_1,q_2$. 

For the rest of the fixed points, we use the same strategy but with consideration of some different
$G$-equivariant complex line bundles. Recall that, as we argued in Example 3.1, it is easily seen that
the existence of the $G$-invariant locally linear $(-1)$-sphere $C$ implies that $X$ is equivariantly homeomorphic to the $G$-Hirzebruch surface $F_1(1,a)$ such that the class $e_1\in H^2(X)$ is 
sent to the class of the $(-1)$-section in $F_1(1,a)$ under the equivariant homeomorphism. 
Furthermore, since the complex conjugation on $\C\P^2$ defines an orientation-preserving involution 
$\tau$ which acts as $-1$ on the second cohomology, it follows easily that, with a further application 
of $\tau$ if necessary, one can arrange to have the class $e_0\in H^2(X)$ sent to the class of the 
$(+1)$-section in $F_1(1,a)$. Consequently, the fiber class $F=e_0-e_1\in H^2(X)$ is sent to the fiber 
class of $F_1(1,a)$ under the equivariant homeomorphism.

With the preceding understood, we consider the $G$-equivariant complex line bundle $L$ on
$X$ defined as follows. Let $\pi: F_1(1,a) \rightarrow B=\s^2$ be the holomorphic $\s^2$-fibration, and
let $F_1,F_2$ be the two invariant fibers containing the fixed points with rotation numbers $(1,\pm a)$,
$(-1,\pm (a+1))$ respectively, and let $b_i=\pi(F_i)\in B$, $i=1,2$. Note that the fixed point $q_i$,
for $i=1,2$, is contained in the preimage of $F_i$ in $X$; denote the other fixed point contained in
the preimage of $F_i$ by $q_i^\prime$, $i=1,2$. 

There is a $G$-equivariant complex line bundle $L^\prime$ on $B$, such that the weight of the 
$G$-action on the fiber $L^\prime_{b_1}$ equals $+1$ and the weight on the fiber $L^\prime_{b_2}$ equals $0$, and $L^\prime$ has degree $1$. With this understood, the $G$-bundle $L$ on $X$ is the pull-back of $\pi^\ast(L^\prime)$ via the equivariant homeomorphism from $X$ to $F_1(1,a)$; it has weight $+1$ on the fibers at $q_1,q^\prime_1$, and weight $0$ on the fibers at $q_2,q^\prime_2$. Moreover, $c_1(L)=F$. 

The virtual dimension of the moduli space of Seiberg-Witten equations associated to $L$ is
$$
d(L)=\frac{1}{n} (c_1(L)^2-c_1(L)\cdot c_1(K_\omega))+ I_{q_1}+I_{q_1^\prime}
=\frac{2}{n}+I_{q_1}+I_{q^\prime_1},
$$
where $ I_{q_1},I_{q_1^\prime}$ are contributions from the fixed points $q_1,q_1^\prime$. There 
are no contributions from $q_2,q^\prime_2$ because the weights of the $G$-action on the 
corresponding fibers of $L$ are zero. 
Since the rotation numbers at $q_1$ are $(1,a)$ and the weight of the $G$-action is $+1$ on $L_{q_1}$,
$$
I_{q_1}=\frac{1}{n} \sum_{x=1}^{n-1} \frac{2(\mu^{x}-1)}{(1-\mu^{-x})(1-\mu^{-ax})}
=\frac{2}{n} \sum_{x=1}^{n-1} \frac{\mu^{x}}{1-\mu^{-ax}}=\frac{n-1-2b}{n},
$$
where $ab=1\pmod{n}$ and $0<b<n$ (cf. \cite{C1}, Example 3.4). Since $(1,-a)$ and $(-1,a)$ are
the same as unordered pairs when $a=1\pmod{n}$, we shall assume without loss of generality that 
$a\neq 1\pmod{n}$ in the calculations below. 

Suppose the rotation numbers at $q_1^\prime$ are $(1,-a)$. Then with the weight of the $G$-action 
being $+1$ on $L_{q^\prime_1}$, we have 
$$
I_{q_1^\prime}=\frac{1}{n} \sum_{x=1}^{n-1} \frac{2(\mu^{x}-1)}{(1-\mu^{-x})(1-\mu^{ax})}
=\frac{2}{n} \sum_{x=1}^{n-1} \frac{\mu^x}{1-\mu^{ax}}=\frac{n-1-2b^\prime}{n},
$$
where $ab^\prime=-1\pmod{n}$ and $0<b^\prime<n$ (cf. \cite{C1}, Example 3.4).
Observing that $b+b^\prime=0\pmod{n}$, we have 
$$
d(L)=\frac{2}{n}+\frac{n-1-2b}{n}+\frac{n-1-2b^\prime}{n}=\frac{2(n-b-b^\prime)}{n}\in\Z.
$$
However, if the rotation numbers at $q_1^\prime$ are $(-1,a)$ instead, we have 
$$
I_{q_1^\prime}=\frac{1}{n} \sum_{x=1}^{n-1} \frac{2(\mu^{x}-1)}{(1-\mu^{x})(1-\mu^{-ax})}
=-\frac{2}{n} \sum_{x=1}^{n-1} \frac{1}{1-\mu^{-ax}}=-\frac{n-1}{n}.
$$
In this case, 
$$
d(L)=\frac{2}{n}+\frac{n-1-2b}{n}-\frac{n-1}{n}=\frac{2(1-b)}{n},
$$
which is non-integral
because $b\neq 1$. Hence the rotation numbers at $q_1^\prime$ must be $(1,-a)$.  A similar argument proves that the rotation numbers at $q_2^\prime$ must be $(-1, -a-1)$, which finishes 
off the proof. 

\end{proof}

\begin{remark}
If we apply Lemma 3.3 to $F_{11}(3,1)$, we see immediately that $F_{11}(3,1)$ contains no
$G$-invariant locally linear $(-1)$-spheres, because the rotation numbers are
$$
(3,1), (3,-1), (4, -1), (4,1).
$$
With $a=3$, we find that the second and the last pairs have the wrong sign.

\end{remark}

With the above preparation on the rotation numbers, we now prove that, by assuming $J$ is generic,
the $J$-holomorphic section $C_0$ of the $\s^2$-fibration in Lemma 2.3(2) must be a $(-1)$-sphere. 

To this end, we recall the formula for the virtual dimension of the moduli space of $J$-holomorphic 
curves in a $4$-orbifold, applied here to the curve $C_0/G$ in the $4$-orbifold $X/G$. Let $f:
\Sigma\rightarrow X/G$ be a $J$-holomorphic parametrization of $C_0/G$, where $\Sigma$ is the 
orbifold Riemann two-sphere with two orbifold points $z_i$, $i=1,2$, of order $n$, and let 
$(\hat{f}_{z_i},\rho_{z_i}): (D_i,\Z_n)\rightarrow (V_i,G)$ be a local representative of $f$ at $z_i$, 
where the action of $\rho_{z_i}(\mu)$ on $V_i$ is given by 
$$
\rho_{z_i}(\mu)\cdot (w_1,w_2) = (\mu^{m_{i,1}}w_1,\mu^{m_{i,2}} w_2),
\mbox{ with } \mu=\exp(2\pi/n), \; 0\leq m_{i,1},m_{i,2}<n.
$$
Then the virtual dimension of the moduli space of $J$-holomorphic curves at 
$C_0/G$ is given by $2d$, where 
$$
d=-\frac{c_1(K_\omega)\cdot C_0}{n}+ 1-\sum_{i=1}^2\frac{m_{i,1}+m_{i,2}}{n}.
$$
See \cite{CR}, Lemma 3.2.4. With this understood, we recall that the transversality theorem 
(cf. \cite{C1}, Lemma 1.10) for the moduli space of $J$-holomorphic curves implies that $d\geq 0$
for a generic $J$. 

For our purpose, we would like to express $d$ in terms of the self-intersection number of $C_0$.
To this end, we apply the adjunction formula (\cite{C}, Theorem 3.1) to $C_0/G$, and as 
$C_0$ is embedded (equivalently, $C_0/G$ is embedded), we obtain
$$
\frac{1}{2n}(C_0^2+c_1(K_\omega)\cdot C_0)+1=2(\frac{1}{2}-\frac{1}{2n})=1-\frac{1}{n},
$$
which gives the desired expression for $d$:
$$
d=\frac{1}{n} C_0^2+\frac{2}{n}+1-\sum_{i=1}^2\frac{m_{i,1}+m_{i,2}}{n}
$$

In order to compute $d$, we need to locate the two fixed points on $C_0$. By looking at the rotation numbers, it is clear that $q_1,q_1^\prime$ and $q_2,q_2^\prime$ are contained in two distinct invariant fibers of the $\s^2$-fibration in Lemma 2.3(2). It follows easily that the following are the only possibilities 
for the fixed points on $C_0$:
$$
\mbox{(a) $q_1,q_2\in C_0$; \;\; (b) $q_1,q_2^\prime\in C_0$; \;\; (c) $q_1^\prime,q_2\in C_0$; \;\;
(d) $q_1^\prime,q_2^\prime\in C_0$.}
$$
With this understood, the weights $(m_{i,1},m_{i,2})$, $i=1,2$, in the formula for $d$ can be read off 
from the rotation numbers (since $C_0$ is embedded and is a section) and are given correspondingly 
as follows:
\begin{itemize}
\item [{(a)}] $(m_{1,1},m_{1,2})=(1,a)$, $(m_{2,1},m_{2,2})=(1, n-a-1)$;
\item [{(b)}] $(m_{1,1},m_{1,2})=(1,a)$, $(m_{2,1},m_{2,2})=(1, a+1)$;
\item [{(c)}] $(m_{1,1},m_{1,2})=(1,n-a)$, $(m_{2,1},m_{2,2})=(1, n-a-1)$;
\item [{(d)}] $(m_{1,1},m_{1,2})=(1,n-a)$, $(m_{2,1},m_{2,2})=(1, a+1)$
\end{itemize}
where $a$ satisfies the inequality $0<a<n-1$. Correspondingly, we have
$$
(a) \;d=\frac{C_0^2+1}{n}, \; (b) \; d=\frac{C_0^2+n-2a-1}{n}, \; (c)\;  
d=\frac{C_0^2+2a+1-n}{n},  \;(d)\;  d=\frac{C_0^2-1}{n}.
$$
If $J$ is generic, we have $d\geq 0$, so that with the fact that $C_0^2<0$, we obtain
$$
(a) \; C_0^2=-1, \; (b) \; C_0^2= -n+2a+1, \; (c)\;  C_0^2=-2a-1+n,
$$ 
and case (d) is a contradiction. Furthermore, since $n$ is odd, (b) and (c) can be ruled out by the 
fact that $C_0^2$ is odd (cf. Lemma 2.3(2)). This shows that $C_0$ is a $G$-invariant 
$J$-holomorphic $(-1)$-sphere when $J$ is generic. Thus Theorem 1.1 is proved for the case of
pseudo-free actions. 

\section{Invariant $(-1)$-spheres of non-pseudo-free actions}

For non-pseudo-free actions, the argument in the previous section broke down in a couple of places. 
One of them is Lemma 3.3, where a theorem of Wilczynski \cite{Wil1} asserting that a pseudo-free 
locally linear action on $\C\P^2$ is equivalent to a linear action was used. Perhaps the most serious 
obstacle in the non-pseudo-free case is the failure of the argument for ruling out the cases (b) and (c) 
at the end of the proof. The argument relies on the fact that the order $n$ of $G$ is odd, which is no 
longer true for a non-pseudo-free action in general. 

In this section, we shall give a different proof for Lemma 3.3 for non-pseudo-free actions, 
and rescue the argument of ruling out (b) and (c) for the case of $n>2$ and even by 
exploiting the smoothness assumption of the invariant $(-1)$-sphere $C$. 

For the first part, we continue to assume that $C$ is a $G$-invariant locally linear 
$(-1)$-sphere in $X$. 
If there exists a $g\in G$ which acts trivially on $C$, then $C$, as a $2$-dimensional fixed component of 
$g$, is naturally a smooth, $\omega$-symplectic $(-1)$-sphere, and Theorems 1.1 and 1.2 are trivially 
true. So without loss of generality, we assume the induced action of $G$ on
$C$ is effective. Then as in the previous section, $C$ contains exactly two fixed points of the $G$-action 
on $X$. We shall orient $C$ so that the class of $C$ equals $e_1$, and with this choice of orientation 
on $C$, the rotation numbers of the $G$-action at the two fixed points contained in $C$, continued to
be denoted by $q_1,q_2$, can be written as unordered pairs $(1, a)$ and $(-1,a+1)$ for some 
$a\in\Z$ mod $n$ after fixing an appropriate generator $\mu\in G$, with the second number in each pair 
standing for the weight in the normal direction. The only difference from the pseudo-free case is that
the weights $a$, $a+1$ are no longer required to be co-prime to $n$, and consequently, $n$ may be 
an even integer in this case. Finally, since the pseudo-free case has been dealt with in the previous
section, we shall consider exclusively the non-pseudo-free case, i.e., there is a $2$-dimensional fixed component $\Sigma$ of some element $\kappa\in G$.

Our first observation is that the induced $G$-action on the base $\s^2$ of the $\s^2$-fibration from
Lemma 2.3(2) must be effective. To see this, suppose to the contrary that there is an element $h\in G$ 
which acts trivially on the base $\s^2$. Then $h$ must fix two $J$-holomorphic sections, denoted by $E_0,E_\infty$, of the $\s^2$-fibration. Furthermore, it is easy to see that every fixed point of $G$ is contained in $E_0$ or $E_\infty$; in particular, $q_1,q_2\in E_0\cup E_\infty$. Now note that the 
weights of the action of $h$ at $q_1,q_2$ can be written as $(k, ka)$ and $(-k, k(a+1))$ for some 
$k\neq 0 \pmod{n}$. However, as $q_1,q_2\in E_0\cup E_\infty$, we must have $ka=k(a+1)=0 \pmod{n}$, which implies $k=0\pmod{n}$, a contradiction. Hence the claim follows. As a consequence, the
$\s^2$-fibration has exactly two $G$-invariant fibers, with $\Sigma$ being one of them; in particular, 
$\Sigma=F=e_0-e_1$. We denote the other $G$-invariant fiber by $\Sigma^\prime$. 

Before we proceed further, recall that the intersection theory works for locally flat, topologically 
embedded surfaces in a topological $4$-manifold (cf. Freedman-Quinn \cite{FQ}). With this 
understood, observe that $C\cdot \Sigma=e_1\cdot F=1$, so that $C\cap \Sigma\neq \emptyset$. 
This implies that either $q_1$ or $q_2$ must be contained in $\Sigma$. Without loss of generality, we assume $q_1\in \Sigma$. Writing $\kappa=\mu^k$ where $\mu$ is the fixed generator of $G$, we see 
that the weights of the action of $\kappa$ at $q_1$ are $(k,ka)$, implying $ka=0\pmod{n}$. Consequently, the weights of the action of $\kappa$ at $q_2$ are $(-k, k(a+1))=(-k,k)$, which implies that $q_2$ is not contained in $\Sigma$. Hence $q_1$ is the only intersection point of $\Sigma$ and $C$. We further 
notice that the action of $\mu$ on the complex vector space $(T_{q_1} X, J)$ has two distinct 
eigenspaces, which are $T_{q_1} C$ and $T_{q_1}\Sigma$, so that $C$ and $\Sigma$ must intersect transversely and positively at $q_1$. (Here we used the fact that $C$ is locally linear.)
It follows easily that with respect to the complex structure 
determined by $J$, the rotation numbers at $q_1$ are $(1,a)$, with the second number $a$ being the 
weight in the fiber direction. Finally, note that $q_2$ must be contained in $\Sigma^\prime$. 

It turns out that for the rest of the arguments, it is more convenient to divide the discussions 
according to the following two scenarios :
\begin{itemize}
\item [{(i)}] Neither $\Sigma$ nor $\Sigma^\prime$ is fixed by $G$, i.e., $0<a< n-1$; in particular, 
$n\neq 2$.
\item [{(ii)}] Either $\Sigma$ or $\Sigma^\prime$ is fixed by $G$, i.e., either $a=0$ or $a=n-1$.
\end{itemize}


{\bf Case (i): Neither $\Sigma$ nor $\Sigma^\prime$ is fixed by $G$}. In this case, each of
$\Sigma$, $\Sigma^\prime$ contains another fixed point of $G$, which we continue to denote 
by $q_1^\prime$, $q_2^\prime$ respectively. Since $\Sigma$ has a trivial normal bundle
in $X$, it follows easily that the rotation numbers at $q_1^\prime$ are $(1,-a)$ with respect to the 
complex structure determined by $J$. It remains to determine the rotation numbers at 
$q_2,q_2^\prime$ with respect to the complex structure determined by $J$. 

Let $\pi: X\rightarrow B=\s^2$ be the $\s^2$-fibration, and set $b=\pi(\Sigma)$, 
$b^\prime=\pi(\Sigma^\prime)\in B$. Note that $B$ has a natural orientation determined by
the orientation of $X$ and the orientation of the fibers given by $J$. With respect to this orientation,
the induced $G$-action on $B$ has weight $+1$ at $b$, so that the weight at $b^\prime$ 
must be $-1$. Consequently, with respect to the complex structure determined by $J$, 
the weight of the $G$-action on $X$ must be $-1$ in the normal direction of the $G$-invariant 
fiber $\Sigma^\prime$. With this understood, we claim that the weight of the $G$-action at $q_2$ 
must be $a+1$ in the fiber direction. To see this, note that the rotation numbers at $q_2$ determined 
by $J$ are either $(-1,a+1)$ or $(1,-a-1)$. Our claim is clear in the former case. Assume the latter is 
true. If $n>2$, then we must have $-a-1=-1\pmod{n}$, which gives $a=0\pmod{n}$, and furthermore, 
the weight in the fiber direction must be $1$, which equals $a+1 \pmod{n}$. If $n=2$, then 
$a=0\pmod{2}$ must be true. Moreover, the weight in the fiber direction is $1$, which can be also 
written as $a+1\pmod{2}$. This shows that the rotation numbers at $q_2$ are $(-1,a+1)$ with the 
second entry being the weight in the fiber direction. Finally, it follows easily that the rotation numbers 
at $q_2^\prime$ are $(-1,-a-1)$. 

\vspace{2mm}

The following lemma summarizes the discussion, which corresponds to Lemma 3.3.

\begin{lemma}
For any $G$-invariant $\omega$-compatible almost complex structure 
$J$ on $X$, the rotation numbers determined by the corresponding complex structure on the 
tangent spaces are $(1,a)$ and $(-1,a+1)$ at $q_1,q_2$, and $(1,-a)$ and $(-1,-a-1)$ at
the other two fixed points $q_1^\prime,q_2^\prime$ respectively. 
\end{lemma}

With this in hand, one can argue similarly as in the pseudo-free case that for a generic 
$G$-invariant $J$, the self-intersection number of $C_0$ falls into three possibilities:
$$
(a) \; C_0^2=-1, \; (b) \; C_0^2= -n+2a+1, \;   \mbox{ or } (c)\; C_0^2=-2a-1+n
$$ 
according to (a) $q_1,q_2\in C_0$, (b) $q_1,q_2^\prime\in C_0$, or (c) $q_1^\prime,q_2\in C_0$. 
We finish off the proof of Theorem 1.1 in case (i) (i.e., neither $\Sigma$ nor $\Sigma^\prime$ is 
fixed by $G$) by observing that $n\neq 2$ so that $n$ is odd.

For Theorem 1.2 (continuing with the assumption that neither $\Sigma$ nor $\Sigma^\prime$ is 
fixed by $G$), we shall only consider the case where $n>2$ and even, and assume that 
the $G$-invariant $(-1)$-sphere $C$ is smoothly embedded. The smoothness assumption on $C$ 
will play an essential role in our argument. First, we observe

\begin{lemma}
Let $n>2$ be an even integer. Suppose a linear $\Z_n$-action on $\R^4$ preserves a complex 
structure $J$ on $\R^4$. Then every $\Z_n$-invariant $2$-dimensional subspace of $\R^4$ is 
$J$-invariant. 
\end{lemma}

\begin{proof}
If the $\Z_n$-action on the complex vector space $(\R^4,J)$ has two distinct eigenspaces (i.e., with
distinct weights), then the corresponding $2$-dimensional real subspaces are the only ones which
are $\Z_n$-invariant. The lemma follows easily in this case.

Suppose the $\Z_n$-action on $(\R^4,J)$ is given by a complex scalar multiplication (i.e., with
the same weight). Then one can choose appropriate coordinates so that a generator $g$ of $\Z_n$ 
and the complex structure $J$ may be represented by matrices 
$$
g=\left (\begin{array}{cccc}
\cos \frac{2\pi}{n} & -\sin\frac{2\pi}{n} & 0 & 0\\
\sin\frac{2\pi}{n} & \cos\frac{2\pi}{n} & 0 & 0\\
0 & 0 &\cos \frac{2\pi}{n} & -\sin\frac{2\pi}{n}\\
0 & 0 & \sin\frac{2\pi}{n} & \cos\frac{2\pi}{n}\\
\end{array}
\right )
\mbox{ and }
J=\left (\begin{array}{cccc}
0 & -1 & 0 & 0\\
1 & 0 & 0 & 0\\
0 & 0 & 0 & -1\\
0 & 0 & 1 & 0\\
\end{array}
\right ).
$$
Let $L$ be any $\Z_n$-invariant $2$-dimensional subspace of $\R^4$. If $n=4m$, then $J=g^m$
so that for any vector $v\in L$, $Jv=g^m \cdot v\in L$. If $n=4m+2$, then for any vector $v\in L$,
$g\cdot v+g^{2m} \cdot v= 2\sin \frac{2\pi}{n}\cdot Jv$. Since $\sin \frac{2\pi}{n}\neq 0$ as $n>2$,
one has $Jv\in L$ as well. This proves that $L$ is $J$-invariant, and the proof of the lemma is
complete.

\end{proof}

As a corollary of Lemma 4.2, we see that in the case of $n>2$ and even, the tangent spaces of 
$C$ at $q_1,q_2$ are $J$-invariant for any $G$-invariant $\omega$-compatible almost complex 
structure $J$ on $X$. Furthermore, since $C$ intersects $\Sigma$ positively, the symplectic
form $\omega$ is positive on $C$ near $q_1$. We claim that the same is true at $q_2$ if the
intersection of $C$ and $\Sigma^\prime$ is transverse. To see this, suppose the intersection of 
$C$ and $\Sigma^\prime$ is transverse and negative at $q_2$. Then note that the $G$-action
on $C\cap \Sigma^\prime\setminus \{q_2\}$ is free so that by a small, equivariant perturbation
of $C$ supported in the complement of $\{q_2\}$, one may assume $C$ and $\Sigma^\prime$ 
intersect transversely. It follows then that
$$
1=e_1\cdot F =C\cdot \Sigma^\prime=-1 \pmod{n},
$$
contradicting the assumption $n>2$. Hence the claim follows.

The above discussion has the following consequence: for any fixed $G$-invariant 
$\omega$-compatible almost complex structure $J_0$ on $X$, we can fix a $G$-invariant, 
$\omega$-compatible almost complex structure $\hat{J}_0$ in a neighborhood of $q_1,q_2$ 
with the following significance:
\begin{itemize}
\item $\hat{J}_0$ agrees with $J_0$ at $q_1,q_2$;
\item $C$ is $\hat{J}_0$-holomorphic near $q_1$, and when $C$ intersects $\Sigma^\prime$ 
transversely, it is also $\hat{J}_0$-holomorphic near $q_2$. 
\end{itemize}

\begin{lemma}
Let $J_0$ be any given $G$-invariant $\omega$-compatible almost complex structure. Then for any 
$G$-invariant $J$ which equals $\hat{J}_0$ in a neighborhood of $q_1,q_2$, the intersection number
of $C$ with the $J$-holomorphic section $C_0$ satisfies the following congruence equation
\begin{itemize}
\item $C\cdot C_0=a \pmod{n}$ if $q_1,q_2^\prime\in C_0$;
\item $C\cdot C_0=-a-1 \pmod{n}$ if $q_1^\prime,q_2\in C_0$.
\end{itemize}
\end{lemma}

\begin{proof}
We first consider the case where $q_1,q_2^\prime\in C_0$. The local intersection number
of $C$ and $C_0$ at $q_1$, both being $\hat{J}_0$-holomorphic near $q_1$, can be determined 
as follows. By work of Micallef and White (cf. \cite{MW}, Theorems 6.1 and 6.2), for any $C$ or 
$C_0$, there is a $C^1$-coordinate chart $\Psi: V\rightarrow \C^2$ near $q_1$, such that the $J$-holomorphic curve may be parametrized by a holomorphic map of the form $z\mapsto (z, f(z))$ 
(here we also use the fact that $C$ and $C_0$ are embedded near $q_1$). Furthermore, since
$C$ and $C_0$ are tangent at $q_1$, there are $f(z), f_0(z)$ such that $C$, $C_0$ are 
parametrized by $z\mapsto (z, f(z))$, $z\mapsto (z, f_0(z))$ with respect to the same chart. 
There are two possibilities: (1) $C$, $C_0$ are distinct near $q_1$, (2) $C\equiv C_0$ near $q_1$. 
In case (1), the local intersection number of $C,C_0$ equals the order of vanishing 
of $f(z)-f_0(z)$ at $z=0$ (cf. \cite{MW}), which equals $a\pmod{n}$ because $C,C_0$ are
$G$-invariant and the weight of the $G$-action in the normal direction equals $a$. In case (2),
the local intersection number is not well-defined. However, we may remedy 
this by performing a small, equivariant perturbation to $C$ so that it is represented by the graph 
of $z\mapsto g(z)$ near $q_1$ for some $g(z)\neq f(z)$. Then the local intersection number 
becomes well-defined and it equals $a\pmod{n}$.

On the other hand, the $G$-action on $C\cap C_0\setminus \{q_1\}$ is free, so that by a small,
equivariant perturbation of $C$ away from $q_1$, one can arrange $C$ and $C_0$ 
intersect transversely, so that the contribution to $C\cdot C_0$ that are from the intersection
points other than $q_1$ equals $0\pmod{n}$. The lemma follows easily for this case.

When $q_1^\prime,q_2\in C_0$, the lemma follows by the same argument if $C$ and 
$\Sigma^\prime$ intersect transversely. Suppose $C$ and $\Sigma^\prime$ are tangent 
at $q_2$. Then the weights of the $G$-action on $T_{q_2}X$ must be the same in all 
complex directions;
in particular, $a+1=-1 \pmod{n}$. On the other hand, since $n>2$ and the weight of the 
$G$-action on $T_{q_2} C$ and $T_{q_2}\Sigma^\prime$ is the same (which is $-1$), the
orientations on $T_{q_2} C$ and $T_{q_2}\Sigma^\prime$ must coincide. Consequently, 
the local intersection number of $C$ and $C_0$ at $q_2$ must be $+1$ because $C_0$ is
a $J$-holomorphic section. Then it follows that $C\cdot C_0=1 \pmod{n}$ as we argued in
the previous case. But this is the same as $C\cdot C_0=-a-1 \pmod{n}$ because 
$a+1=-1 \pmod{n}$. Hence the lemma follows. 

\end{proof}

The following lemma finishes off the proof of Theorem 1.2 in case (i), i.e., 
neither $\Sigma$ nor $\Sigma^\prime$ is fixed by $G$.

\begin{lemma}
For any fixed $G$-invariant $J_0$, the $J$-holomorphic section $C_0$ is a $(-1)$-sphere for any
generic $G$-invariant $J$ which equals $\hat{J}_0$ in a neighborhood of $q_1,q_2$. 
\end{lemma}

\begin{proof}
Let $\delta=C\cdot C_0$. Then the fact that $C=e_1$, $F=e_0-e_1$ and $C_0\cdot F=1$ implies
that $C_0=(\delta +1) e_0-\delta e_1$, which then gives $C_0^2=2\delta+1$. By Lemma 4.3,
$\delta=a \pmod{n}$ if $q_1,q_2^\prime\in C_0$, and $\delta=-a-1\pmod{n}$ 
if $q_1^\prime,q_2\in C_0$. Consequently, $C_0^2=2a+1\pmod{2n}$ in the former case, and
$C_0^2=-2a-1\pmod{2n}$ in the latter case.

Now we recall that the transversality theorem for moduli space of $J$-holomorphic curves continues 
to hold even if we consider a more restrictive class of $G$-invariant almost complex structures $J$,
i.e., those which equal to $\hat{J}_0$ in a fixed neighborhood of $q_1,q_2$, because no 
$J$-holomorphic curves can lie completely in such a neighborhood. Consequently, for a generic 
such $J$, we continue to have $C_0^2= -n+2a+1$ if $q_1,q_2^\prime\in C_0$ and 
$C_0^2=-2a-1+n$
if $q_1^\prime,q_2\in C_0$. But this contradicts what was obtained in the previous paragraph, 
hence $C_0^2=-1$ must be true. This finishes off the proof. 

\end{proof}

{\bf Case (ii): Either $\Sigma$ or $\Sigma^\prime$ is fixed by $G$}. Without loss of generality, we shall
only consider the case where $\Sigma$ is fixed, which corresponds to $a=0$; the other case is 
completely analogous.

We note that in this case, the fixed point $q_1^\prime$ is not well-defined. However, $q_2^\prime$
is well-defined and the rotation numbers at $q_2$, $q_2^\prime$ continue to be $(-1, a+1)=(-1,1)$,
$(-1,-a-1)=(-1,-1)$ respectively (with respect to almost complex structure $J$).

Let $x,y$ be the two fixed points on $C_0$, where $x\in \Sigma$, and $y=q_2$ or $q_2^\prime$. 
Then it follows easily that the weights $(m_{1,1},m_{1,2})$ at $x$ equals $(1,0)$ and the weights
$(m_{2,1},m_{2,2})$ at $y$ equals $(1, n-1)$ or $(1, 1)$, depending on whether $y=q_2$ or 
$q_2^\prime$. Correspondingly, when $J$ is a generic $G$-invariant almost complex structure,
we have 
$$
C_0^2=-1 \mbox{ if } y=q_2, \mbox{ and } C_0^2=-n+1 \mbox{ if } y=q_2^\prime. 
$$
Theorem 1.1 follows immediately, observing that if $n=2$ or $n$ is odd, $C_0^2$ must be $-1$. 

For Theorem 1.2 where $n>2$ and even, we can eliminate the case where $y=q_2^\prime$ as 
follows. If $x=q_1$, then Lemma 4.3 still applies and we get $C\cdot C_0=a=0 \pmod{n}$. If $x\neq
q_1$, then $C\cdot C_0=0 \pmod{n}$ holds automatically. Then, in any event, we have, as in 
Lemma 4.4, $C_0^2=1\pmod{2n}$ which contradicts $C_0^2=-n+1$. This finishes off the proof
for Theorem 1.2. 

\section{Smooth classification of $G$-Hirzebruch surfaces}

We fix a generator $\mu\in G=\Z_n$, and let $F_r(a,b)$ be the Hirzebruch surface $F_r$ equipped with
a homologically trivial, holomorphic $G$-action with the following fixed-point set structure. Note that such
a $G$-action has two invariant fibers, denoted by $F_0, F_1$, and leaves the zero-section $E_0$ and
the infinity-section $E_1$ invariant also. (Here our convention is from \cite{Wil2} that $E_0\cdot E_0=-r$.)
We set $x_{ij}=F_i\cap E_j$ for $i,j=0,1$, which are fixed points of the $G$-action. With this understood,
the integers $(a,b) \pmod{n}$ are the rotation numbers at $x_{00}$ with respect to the complex 
structure on $F_r$, with the second number in the pair being the weight of the action in the fiber 
direction. The rotation numbers at the other fixed points $x_{01}, x_{10},x_{11}$ are $(a,-b)$, $(-a,b+ra)$, and $(-a,-b-ra)$. See Wilczynski \cite{Wil2}, Theorem 4.1. We note that the integers $a,b,n$ must satisfy 
$\gcd(a,b,n)=1$, and furthermore, 
\begin{itemize}
\item $\gcd(a,n)=1$ if and only if $E_0,E_1$ have trivial isotropy;
\item $\gcd(b,n)=1$ if and only if $F_0$ has trivial isotropy; and
\item $\gcd(b+ra,n)=1$ if and only if $F_1$ has trivial isotropy.
\end{itemize}

Let $F_{r^\prime}(a^\prime,b^\prime)$ be another $G$-Hirzebruch surface with the 
corresponding invariant fibers and sections and the fixed points denoted by $F_i^\prime$,
$E_i^\prime$, and $x_{ij}^\prime$ respectively. Under appropriate numerical conditions 
on $(a,b,r)$ and $(a^\prime,b^\prime,r^\prime)$, there are six types, $c_1, c_2, \cdots, c_6$, 
of canonically defined, orientation-preserving, equivariant diffeomorphisms between $F_r(a,b)$ and 
$F_{r^\prime}(a^\prime,b^\prime)$, which we describe below. (See also related discussions in
Wilczynski \cite{Wil2}.)

{\it Type $c_1$}. Suppose $a^\prime=-a$, $b^\prime=-b$, and $r^\prime=r$. Then there is an
equivariant diffeomorphism $c_1:F_r(a,b)\rightarrow F_{r^\prime}(a^\prime,b^\prime)$, which
sends $F_i$ to $F^\prime_i$, $E_i$ to $E_i^\prime$, and which induces $z\mapsto \bar{z}$ 
between the bases and the fibers of $F_r$ and $F_{r^\prime}$.

{\it Type $c_2$}. Suppose $a^\prime=-a$, $b^\prime=b+ra$, and $r^\prime=r$. Then there is an
equivariant diffeomorphism $c_2:F_r(a,b)\rightarrow F_{r^\prime}(a^\prime,b^\prime)$, which
sends $F_0$ to $F^\prime_1$ and $F_1$ to $F_0^\prime$, $E_i$ to $E_i^\prime$, and which induces $z\mapsto {z}^{-1}$ between the bases of $F_r$ and $F_{r^\prime}$.

{\it Type $c_3$}. Suppose $a^\prime=a$, $b^\prime=-b$, and $r^\prime=-r$. Then there is an
equivariant diffeomorphism $c_3:F_r(a,b)\rightarrow F_{r^\prime}(a^\prime,b^\prime)$, which
sends $F_i$ to $F^\prime_i$, $E_0$ to $E_1^\prime$, $E_1$ to $E_2^\prime$, and which 
induces $z\mapsto {z}^{-1}$ between the fibers of $F_r$ and $F_{r^\prime}$.

{\it Type $c_4$}. Suppose $r^\prime=r=0$, $a^\prime=b$, and $b^\prime=a$. Then there is an
equivariant diffeomorphism $c_4:F_r(a,b)\rightarrow F_{r^\prime}(a^\prime,b^\prime)$, which
switches the fibers and sections between $F_r$ and $F_{r^\prime}$, sending $F_0$ to 
$E_0^\prime$, $F_1$ to $E_1^\prime$, and  $E_0$ to $F_0^\prime$, $E_1$ to $F_1^\prime$.

\vspace{1.5mm}

For the types $c_5,c_6$, we assume that $\gcd(a,n)=\gcd(a^\prime,n)=1$.

\vspace{1.5mm}

{\it Type $c_5$}. Suppose $a^\prime=a$, $b^\prime=b$, and $r^\prime =r\pmod{2n}$. Then there is an
equivariant diffeomorphism $c_5:F_r(a,b)\rightarrow F_{r^\prime}(a^\prime,b^\prime)$, sending the fixed
points $x_{ij}$ to $x_{ij}^\prime$, $i,j=0,1$. To see this, we shall explain that there is a diffeomorphism 
between the quotient orbifolds, $\hat{c}_5: F_r(a,b)/G\rightarrow F_{r^\prime}(a^\prime,b^\prime)/G$, 
which are orbifold $\s^2$-bundles over an orbifold $\s^2$ with two singular points $z_0,z_1$ of order $n$.
Moreover, the orbifold $\s^2$-bundles are induced, under the canonical embedding $\s^1\subset SO(3)$, from the principal orbifold $\s^1$-bundles of Euler number $-r/n$, $-r^\prime/n$ respectively. Since
$\gcd(a,n)=\gcd(a^\prime,n)=1$, we may assume without loss of generality that $a=a^\prime=1$. Then
the Seifert invariants of the two bundles, which are the same, are $(n,\beta_0)$, $(n,\beta_1)$ at 
$z_0,z_1$ where $\beta_0=b\pmod{n}$, $\beta_1=-b-r\pmod{n}$. With this understand, let 
$e,e^\prime\in\Z$ such that
$$
-\frac{r}{n}=\frac{\beta_0}{n}+\frac{\beta_1}{n}+e, \;\;
-\frac{r^\prime}{n}=\frac{\beta_0}{n}+\frac{\beta_1}{n}+e^\prime, 
$$
then $r^\prime-r=(e-e^\prime)\cdot n$, which implies $e^\prime=e\pmod{2}$ because 
$r^\prime =r\pmod{2n}$. Since $\pi_1 SO(3)=\Z_2$, $e^\prime=e\pmod{2}$ implies that the
two orbifold $\s^2$-bundles are isomorphic, which gives the diffeomorphism 
$\hat{c}_5: F_r(a,b)/G\rightarrow F_{r^\prime}(a^\prime,b^\prime)/G$. 

{\it Type $c_6$}. Suppose $a^\prime=a$, $b^\prime=b$, and $r^\prime a^\prime  =-2b-ra\pmod{2n}$.
Then there is an equivariant diffeomorphism $c_6:F_r(a,b)\rightarrow F_{r^\prime}(a^\prime,b^\prime)$, sending the fixed points $x_{0j}$ to $x_{0j}^\prime$, $j=0,1$, and $x_{10}$ to $x_{11}^\prime$,
$x_{11}$ to $x_{10}^\prime$. To see this, note that switching $x_{10}$ and $x_{11}$ means applying 
$z\mapsto z^{-1}$ to a neighborhood of $F_1$, which has the effect of changing the sign of $\beta_1$
in the Seifert invariant. Therefore, there is a diffeomorphism from $F_r(a,b)$ to $F_{\tilde{r}}(a,b)$
which switches $x_{10}$ and $x_{11}$, where $\tilde{r}$ satisfies 
$$
-\frac{\tilde{r}}{n}=\frac{\beta_0}{n}+\frac{-\beta_1}{n}+e.
$$
It follows that $\tilde{r}-r=2\beta_1=-2(b+r) \pmod{2n}$, which gives $\tilde{r}=-2b-r \pmod{2n}$.
Note that $r^\prime=\tilde{r}\pmod{2n}$, so that there is a $c_5:F_{\tilde{r}}(a,b)\rightarrow 
F_{r^\prime}(a^\prime,b^\prime)$. Consequently, there is an equivariant diffeomorphism 
$c_6:F_r(a,b)\rightarrow F_{r^\prime}(a^\prime,b^\prime)$ as claimed. (Compare also the
relevant discussions in Wilczynski \cite{Wil2}.)

With the preceding preparations, we shall derive in the next two lemmas a set of numerical conditions
which must be satisfied by the triples $(a,b,r)$ and $(a^\prime,b^\prime,r^\prime)$ (modulo the relations
from the canonical equivariant diffeomorphisms $c_1$ through $c_6$) if there is an orientation-preserving equivariant diffeomorphism between $F_r(a,b)$ and $F_{r^\prime}(a^\prime,b^\prime)$.

\begin{lemma}
Suppose $\gcd(a,n)=\gcd(a^\prime,n)=1$. If $F_r(a,b)$ is orientation-preservingly equivariantly 
diffeomorphic to $F_{r^\prime}(a^\prime,b^\prime)$, then after replacing $F_r(a,b)$, 
$F_{r^\prime}(a^\prime,b^\prime)$ by a $G$-Hirzebruch surface (continuously denoted by
$F_r(a,b)$, $F_{r^\prime}(a^\prime,b^\prime)$ for simplicity) which is equivariantly diffeomorphic 
to $F_r(a,b)$, $F_{r^\prime}(a^\prime,b^\prime)$ by a sequence of canonical equivariant 
diffeomorphisms of types $c_i$, $1\leq i\leq 6$, one of the following must be true: (i) 
$F_r(a,b)=F_{r^\prime}(a^\prime,b^\prime)$, or (ii) $a^\prime=a$, $b^\prime=b$, and $r^\prime=r-n$, 
or (iii) $a^\prime=b$, $b^\prime=a$ and $r^\prime=r=n$, where $a\neq \pm b$.
\end{lemma}

\begin{proof}
First, consider the case where $r\neq 0\pmod{n}$ and $2b+ra\neq 0\pmod{n}$. The key observation is
that in this case, $a$ is the unique number among $a,b$ such that either $a$ or $-a$ shows up in all 
four pairs of the rotation numbers, i.e., $(a,\pm b)$ and $(-a,\pm (b+ra))$. It follows easily from the
assumption that $F_r(a,b)$ is orientation-preservingly equivariantly diffeomorphic to 
$F_{r^\prime}(a^\prime,b^\prime)$ that $r^\prime\neq 0\pmod{n}$ and 
$2b^\prime+r^\prime a^\prime\neq 0\pmod{n}$ must also hold, and that $a^\prime=\pm a$.
Furthermore, observe that either $b\neq 0\pmod{n}$ or $b+ra\neq 0\pmod{n}$, which means that at most
one of $F_0,F_1$ is fixed under the $G$-action. We assume without loss of generality that $b\neq 0
\pmod{n}$. Then $x_{00}$ and $x_{01}$, being isolated fixed points, must be sent to the fixed-points 
$x_{ij}^\prime$ for some $i,j$ under the equivariant
diffeomorphism. After replacing $F_{r^\prime}(a^\prime,b^\prime)$ by a $G$-Hirzebruch surface 
(continue to be denoted by $F_{r^\prime}(a^\prime,b^\prime)$ for simplicity) which is equivariantly diffeomorphic to $F_{r^\prime}(a^\prime,b^\prime)$ by a sequence of canonical equivariant 
diffeomorphisms of types $c_i$, $i=2,3$, one can arrange so that $x_{00}$ is sent to $x_{00}^\prime$ 
under the equivariant diffeomorphism from $F_r(a,b)$ to $F_{r^\prime}(a^\prime,b^\prime)$. Then the assumption $r\neq 0\pmod{n}$ and $2b+ra\neq 0\pmod{n}$ implies that $x_{01}$ must be sent to 
$x_{01}^\prime$. With an application of $c_1$ if necessary, we may arrange to have $a^\prime=a$, 
$b^\prime=b$. Finally, if $b+ra=0\pmod{n}$, then $b^\prime +r^\prime a^\prime=0\pmod{n}$ must also
be true, and if $b+ra\neq 0\pmod{n}$, with an application of $c_6$ if necessary, we may arrange to 
have $b+ra=b^\prime +r^\prime a^\prime \pmod{n}$. In any event, $r^\prime=r\pmod{n}$ is satisfied. 
With a further application of $c_5$, we have either $F_r(a,b)=F_{r^\prime}(a^\prime,b^\prime)$,
or $r^\prime=r-n$. Note that when $n\neq 2$ and $x_{1j}$ are isolated, $x_{1j}$ is sent to
$x_{1j}^\prime$ under the equivariant diffeomorphism from $F_r(a,b)$ to 
$F_{r^\prime}(a^\prime,b^\prime)$.

Suppose $r= 0\pmod{n}$ or $2b+ra=0\pmod{n}$. Then $r^\prime= 0\pmod{n}$ or 
$2b^\prime+r^\prime a^\prime=0\pmod{n}$ must also hold. With an application of $c_6$ to both
$F_r(a,b)$ and $F_{r^\prime}(a^\prime,b^\prime)$ if necessary, one may assume 
$r^\prime=r=0\pmod{n}$. If $b=0\pmod{n}$, then we must also have $b^\prime=0\pmod{n}$, and with
a further application of $c_1$, we may arrange to have $a=a^\prime$, $b=b^\prime$. 
If $b\neq 0\pmod{n}$, then $\{x_{ij}\}$ are the only fixed points, and will be sent to $\{x_{ij}^\prime\}$ 
under the equivariant diffeomorphism. With an application of $c_2,c_3$ to 
$F_{r^\prime}(a^\prime,b^\prime)$ if necessary, one can arrange to have $x_{00}$ sent to 
$x_{00}^\prime$ and have $(a,b)=\pm (a^\prime,b^\prime)$ as unordered pairs. Assume first that 
$a\neq \pm b$. Then with an application of $c_1$ if necessary, we have either $a^\prime=b$, 
$b^\prime=a$, in which case $x_{01}$ is sent to $x_{10}^\prime$, or $a^\prime=a$, $b^\prime=b$, 
in which case $x_{01}$ is sent to $x_{01}^\prime$. In the former case, 
if $r=0\pmod{2n}$ or $r^\prime=0\pmod{2n}$, one may apply $c_4$ and $c_5$ to arrange to have 
$a^\prime=a$, $b^\prime=b$. In the latter case, note that $x_{1j}$ is sent to $x_{1j}^\prime$ 
when $n\neq 2$. Assume $a=\pm b$. Then with an application of $c_1$ if necessary,
we may arrange to have $a^\prime=a$, $b^\prime=b$. 
It follows easily that we either have $F_r(a,b)=F_{r^\prime}(a^\prime,b^\prime)$, or $a^\prime=a$, 
$b^\prime=b$ with $r^\prime=r-n$, or $a^\prime=b$, $b^\prime=a$ and $r^\prime=r=n$ with
$a\neq \pm b$. 

\end{proof}

We remark that from the proof it is clear that when $F_r(a,b)\neq 
F_{r^\prime}(a^\prime,b^\prime)$, there is an orientation-preserving equivariant diffeomorphism 
from $F_r(a,b)$ to $F_{r^\prime}(a^\prime,b^\prime)$, which, in the case of $a^\prime=a$, 
$b^\prime=b$ and $r^\prime=r-n$, sends $F_i$ to $F_i^\prime$ if the fibers are fixed, and 
sends $x_{ij}$ to $x_{ij}^\prime$ when $n\neq 2$ if the fixed points are isolated. 
Moreover, in the case of $a^\prime=b$, $b^\prime=a$ and $r^\prime=r=n$, where $a\neq \pm b$, 
the equivariant diffeomorphism sends $x_{ij}$ to $x_{ij}^\prime$ for $i=j$ and sends 
$x_{ij}$ to $x_{ji}^\prime$ for $i\neq j$.

\begin{lemma}
Suppose $\gcd(a,n)\neq 1$. If $F_r(a,b)$ is orientation-preservingly equivariantly diffeomorphic to
$F_{r^\prime}(a^\prime,b^\prime)$, then after replacing $F_{r^\prime}(a^\prime,b^\prime)$ by
a $G$-Hirzebruch surface (continuously denoted by $F_{r^\prime}(a^\prime,b^\prime)$ for simplicity)
which is equivariantly diffeomorphic to $F_{r^\prime}(a^\prime,b^\prime)$
by a sequence of canonical equivariant diffeomorphisms of types $c_i$, $1\leq i\leq 6$, one has
either $F_r(a,b)=F_{r^\prime}(a^\prime,b^\prime)$, or $r=0$, $\gcd(a^\prime,n)=1$ and 
$r^\prime=0\pmod{n}$. 
\end{lemma}

\begin{proof}
First, consider the case where $r\neq 0$. Then since $\gcd(a,n)\neq 1$, the sections $E_0, E_1$
of $F_r(a,b)$ have nontrivial isotropy. On the other hand, $E_0,E_1$ have nonzero self-intersections,
so that they can not be mapped to fibers of $F_{r^\prime}(a^\prime,b^\prime)$ with nontrivial isotropy.
Consequently, the sections $E_0^\prime,E_1^\prime$ of $F_{r^\prime}(a^\prime,b^\prime)$
must also have nontrivial isotropy, and furthermore, $r^\prime=\pm r$ and $a^\prime=\pm a$. 
On the other hand, observe that $b\neq 0\pmod{n}$ and $b+ra\neq 0\pmod{n}$ because 
$\gcd(a,n)\neq 1$ but $\gcd(a,b,n)=1$. (In other words, there are no fixed fibers.) Finally, note that 
$c_2$, $c_3$ act transitively on the set $\{x_{ij}\}$, so that with an application of $c_2,c_3$ to
$F_{r^\prime}(a^\prime,b^\prime)$ if necessary, we may assume that the equivariant 
diffeomorphism from $F_r(a,b)$
to $F_{r^\prime}(a^\prime,b^\prime)$ sends $x_{00}$ to $x_{00}^\prime$. This particularly implies
$(a^\prime,b^\prime)=\pm (a,b)$ as unordered pairs. With a further application of $c_1$, we obtain
$(a^\prime,b^\prime)=(a,b)$ as ordered pairs (note that $a\neq \pm b$ because otherwise, $\gcd(a,b,n)
\neq 1$, and that $a^\prime=\pm a$). Finally, $r^\prime=r$, because $E_0$ must be sent to $E_0^\prime$, 
hence $F_r(a,b)=F_{r^\prime}(a^\prime,b^\prime)$ if $r\neq 0$.

Next, we assume $r=0$. If $\gcd(a^\prime,n)\neq 1$, then $r^\prime$ must be $0$, so that with an
application of $c_4$ if necessary, we may assume the equivariant diffeomorphism from $F_r(a,b)$
to $F_{r^\prime}(a^\prime,b^\prime)$ sends sections to sections. It follows easily that
$F_r(a,b)=F_{r^\prime}(a^\prime,b^\prime)$ up to a sequence of canonical equivariant 
diffeomorphisms.

Finally, consider the case where $r=0$ and $\gcd(a^\prime,n)=1$. We observe that 
$b^\prime+r^\prime a^\prime=\pm b^\prime \pmod{n}$ must be true because it is true for $F_r(a,b)$,
and $F_r(a,b)$ is equivariantly diffeomorphic to $F_{r^\prime}(a^\prime,b^\prime)$. Consequently,
either $r^\prime a^\prime=0 \pmod{n}$, or $2b^\prime+r^\prime a^\prime=0\pmod{n}$. In the latter
case, an application of $c_6$ will reduce it to the former case, which is equivalent to $r^\prime=0
\pmod{n}$. This proves the lemma. 

\end{proof}

Note that in the latter case of Lemma 5.2, we can apply $c_4$ to $F_r(a,b)$, and with a further 
application of $c_5$ to $F_{r^\prime}(a^\prime,b^\prime)$ if necessary, we may arrange to have
either $F_r(a,b)=F_{r^\prime}(a^\prime,b^\prime)$, or $a^\prime=a$, $b^\prime=b$, $r^\prime=r-n$,
and $\gcd(a,n)=1$. Furthermore, there is an orientation-preserving equivariant diffeomorphism 
from $F_r(a,b)$ to $F_{r^\prime}(a^\prime,b^\prime)$, which sends $F_i$ to $F_i^\prime$ if
the fibers are fixed, and sends $x_{ij}$ to $x_{ij}^\prime$ when $n\neq 2$ if the fixed points 
are isolated. 

With the preceding understood, the main technical result is summarized in the following proposition.

\begin{proposition}
Suppose the $G$-actions are non-pseudo-free and $n\neq 2$ unless $r,r^\prime$ are even.
Then there are no orientation-preserving equivariant diffeomorphisms from $F_r(a,b)$ to
$F_{r^\prime}(a^\prime,b^\prime)$, which send $F_i$ to $F_i^\prime$ if the fibers are fixed, 
and send $x_{ij}$ to $x_{ij}^\prime$ when $n\neq 2$ if the fixed points are isolated, where 
$a^\prime=a$, $b^\prime=b$, $r^\prime=r-n$, and $\gcd(a,n)=\gcd(a^\prime,n)=1$. 
\end{proposition}

\begin{remark}
The traditional method for distinguishing non-free smooth finite group actions on a four-manifold
has been through the knotting of the $2$-dimensional fixed-point set of the actions 
(cf. \cite{Gif, Go, HLM, FSS, KR}). Here we give examples of $G$-Hirzebruch surfaces which are not
equivariantly diffeomorphic by Proposition 5.3, but can not be distinguished using the traditional method. 

Consider $F_1(1,3)$ and $F_5(1,3)$ with $G=\Z_4$. The actions are semi-free, and the 
fixed-point set of each $G$-surface consists of two isolated fixed points and one fixed fiber.
The complement of the fixed fiber in each $G$-surface is a trivial complex line bundle over
$\s^2$. It is easily seen that the $G$-actions on each complement are equivariantly diffeomorphic, i.e., there is no ``knotting" of the $2$-dimensional fixed-point set. 

\end{remark}

\noindent{\bf Proof of Proposition 5.3}

\vspace{3mm}

Suppose to the contrary, there is such an orientation-preserving equivariant diffeomorphism 
$f: F_r(a,b)\rightarrow F_{r^\prime}(a^\prime,b^\prime)$. We first observe that $r,r^\prime$ must 
have the same parity, so that $n$ is necessarily even. Without loss of generality,
we assume $a^\prime=a=1$. Furthermore, with an application of $c_3$ to both $F_r(a,b)$ and
$F_{r^\prime}(a^\prime,b^\prime)$ if necessary, we may assume $0\leq 2b^\prime=2b\leq n$.

Next, we claim that with an application of $c_5,c_6$ to both $G$-Hirzebruch surfaces if necessary,
one can arrange so that $r^\prime$ and $r$ satisfy the following constraints:
$$
0\leq r^\prime<n, \;\; b+r^\prime\leq n,\;\; \mbox{ and } n\leq r=r^\prime+n<2n.
$$
To see this, note that with $c_5$, we may assume $0\leq r,r^\prime<2n$, and assuming without loss
of generality that $r^\prime<r$, we have $0\leq r^\prime<n\leq r=r^\prime+n<2n$. If $b+r^\prime\leq n$, 
then we are done. Suppose $b+r^\prime>n$. Then we apply $c_6$ to both $G$-Hirzebruch surfaces 
and replace $r, r^\prime$ by $\tilde{r}=4n-2b-r$ and $\tilde{r}^\prime=2n-2b-r^\prime$ respectively. 
Note that $\tilde{r}^\prime-\tilde{r}=-2n-(r^\prime-r)=-n$, so that $\tilde{r}^\prime=\tilde{r}-n$ continues 
to hold. We will show that the conditions $0<\tilde{r}^\prime$ and $b+\tilde{r}^\prime<n$ are satisfied, 
with which $n\leq \tilde{r}=\tilde{r}^\prime+n<2n$ follows easily. With this understood, 
$0< \tilde{r}^\prime$ follows from $2b\leq n$ and $r^\prime<n$, and $b+\tilde{r}^\prime<n$
follows from the assumption $b+r^\prime>n$. Hence the claim. 

With the preceding preparation, we shall denote $F_r(a,b)$ by $X$; in particular, $X$ is either
$\C\P^2\# \overline{\C\P^2}$ or $\s^2\times\s^2$.  Let $C\subset X$ be the 
pre-image of the holomorphic $(-r^\prime)$-section $E_0^\prime\subset F_{r^\prime}(a^\prime,b^\prime)$ under $f$. Then $C$ is a $G$-invariant, smoothly embedded two-sphere in $X$ with self-intersection
$-r^\prime$. Let $J_0$ be the $G$-invariant complex structure on $X=F_r(a,b)$, and let $\omega_0$ 
be a fixed $G$-invariant K\"{a}hler form. Our goal is to first show that for a certain generic $G$-invariant,
$\omega_0$-compatible $J$, there is a $G$-invariant, embedded $J$-holomorphic two-sphere $\tilde{C}$
with self-intersection $-r^\prime$. On the other hand, by choosing a sequence of such $J$ which 
converges to $J_0$, the corresponding $J$-holomorphic $(-r^\prime)$-spheres will converge to a
cusp-curve $C_\infty$ by Gromov compactness. Carefully analyzing $C_\infty$ will lead to a 
contradiction to the complex geometry of $J_0$, which proves that $f$ should not exist. 

We begin our proof by giving an orientation to $C$. Since the $G$-action is non-pseudo-free, it
follows easily that either $F_0$ or $F_1$ has nontrivial isotropy. Without loss of generality, we assume
$F_0$ has nontrivial isotropy. Then it follows that under $f$, $F_0$ is mapped to $F_0^\prime$.
As a consequence, we see that $C$ intersects $F_0$ transversely. With this understood, we shall
orient $C$ so that $C\cdot F_0=1$. 

Before we proceed further, we shall fix the following notations which are compatible with those used
in Lemmas 2.3 and 2.4. For $X=\C\P^2\#\overline{\C\P^2}$, let $e_0,e_1\in H^2(X)$ be a basis 
such that $c_1(K_{\omega_0})=-3e_0+e_1$; in particular, $F_0=F_1=e_0-e_1$. For $X=\s^2\times
\s^2$, we choose a basis $e_1,e_2\in H^2(X)$ such that $c_1(K_{\omega_0})=-2e_1-2e_2$, and
moreover, $F_0=F_1=e_2$. 

\vspace{2mm}

As in Section 4, it is more convenient to divide our discussions according to the following two scenarios :
\begin{itemize}
\item [{(i)}] Neither $F_0$ nor $F_1$ is fixed by $G$, i.e., $0<b<n-r^\prime$; in particular, $n\neq 2$.
\item [{(ii)}] Either $F_0$ or $F_1$ is fixed by $G$, i.e., either $b=0$ or $b+r^\prime=0$ or $n$.
\end{itemize}

\vspace{2mm}

{\bf Case (i): Neither $F_0$ nor $F_1$ is fixed by $G$}. In this case, $f$ sends $x_{ij}$ to $x_{ij}^\prime$
for all $i,j$; in particular, $C$ contains $x_{00}$ and $x_{01}$. Furthermore, we remark that since 
the space of $G$-invariant $\omega_0$-compatible $J$ is contractible, the rotation numbers at $x_{0j}$, 
$x_{1j}$, which are $(1,\pm b)$, $(-1, \pm (b+r^\prime))$ respectively respect to $J_0$, remain 
unchanged with respect to any other $\omega_0$-compatible $J$. In particular, the weight of the action
on $C$ is $+1$, $-1$ at $x_{00}$ and $x_{01}$ respectively. 

Since $n>2$ and even, Lemma 4.2 is true so that the tangent space of $C$ at $x_{0j}$ is 
$J_0$-invariant for $j=0,1$. As in Section 4, one can construct a $G$-invariant, 
$\omega_0$-compatible, almost complex structure $\hat{J}_0$ in a neighborhood of $x_{00}$,
$x_{10}$, such that 
\begin{itemize}
\item $\hat{J}_0$ agrees with $J_0$ at $x_{00},x_{10}$;
\item for any $i=0,1$, if $C$ intersects the $G$-invariant fiber $F_i$ transversely at $x_{i0}$, then
$C$ is $\hat{J}_0$-holomorphic near $x_{i0}$. 
\end{itemize}

The following lemma is a generalization of Lemma 4.4. 

\begin{lemma}
For any generic $G$-invariant $J$ equaling $\hat{J}_0$ in a neighborhood of $x_{00},x_{10}$, 
there is an embedded $G$-invariant $J$-holomorphic two-sphere $\tilde{C}$ containing $x_{00}$, 
$x_{10}$ such that $\tilde{C}$ and $C$ are homologous; in particular, $\tilde{C}^2=-r^\prime$. 
\end{lemma}

\begin{proof}
For any given $G$-invariant $J$ equaling $\hat{J}_0$ in a neighborhood of $x_{00},x_{10}$, 
we apply Lemmas 2.3 and 2.4 to it. We further note that, in the present situation, since $F_0$ 
has nontrivial isotropy, it is automatically $J$-holomorphic, so that the $\s^2$-fibration structure
on $X$ always exists, and we may consider $C_0$ as a $J$-holomorphic section even in case (1)
of the lemmas. 

As we have seen before, there are four possibilities for the fixed points on $C_0$:
$$
\mbox{(a) $x_{00},x_{10}\in C_0$; \;\; (b) $x_{00},x_{11}\in C_0$; \;\; (c) $x_{01},x_{10}\in C_0$; \;\;
(d) $x_{01},x_{11}\in C_0$.}
$$
Moreover, the weights $(m_{i,1},m_{i,2})$, $i=1,2$, in the formula for the virtual dimension $2d$ 
of the moduli space of $J$-holomorphic curves at $C_0$ can be read off from the rotation numbers 
and are given correspondingly as follows:
\begin{itemize}
\item [{(a)}] $(m_{1,1},m_{1,2})=(1,b)$, $(m_{2,1},m_{2,2})=(1, n-b-r^\prime)$;
\item [{(b)}] $(m_{1,1},m_{1,2})=(1,b)$, $(m_{2,1},m_{2,2})=(1, b+r^\prime)$;
\item [{(c)}] $(m_{1,1},m_{1,2})=(1,n-b)$, $(m_{2,1},m_{2,2})=(1, n-b-r^\prime)$;
\item [{(d)}] $(m_{1,1},m_{1,2})=(1,n-b)$, $(m_{2,1},m_{2,2})=(1, b+r^\prime)$
\end{itemize}
where $b,r^\prime$ satisfy the inequalities $0<b<n-r^\prime$. Correspondingly, we have
$$
(a) \;d=\frac{C_0^2+r^\prime}{n}, \; (b) \; d=\frac{C_0^2+n-2b-r^\prime}{n}, \; (c)\;  
d=\frac{C_0^2+2b+r^\prime-n}{n}, \;
(d)\;  d=\frac{C_0^2-r^\prime}{n}.
$$
We choose $J$ to be a generic $G$-invariant almost complex structure equaling $\hat{J}_0$ in a
neighborhood of $x_{00},x_{10}$. Then $d\geq 0$, so that with $C_0^2\leq 0$, we obtain
$$
(a) \; C_0^2=-r^\prime, \; (b) \; C_0^2= -n+2b+r^\prime, \; (c)\;  C_0^2=-2b-r^\prime+n,
$$ 
and case (d) is a contradiction unless $r^\prime=0$, and in this case $C_0^2=0$. 

In order to rule out cases (b), (c), we observe that Lemma 4.3 continues to hold here, i.e., 
$C\cdot C_0=b \pmod{n}$ if $x_{00}, x_{11}\in C_0$ and 
$C\cdot C_0=-b-r^\prime \pmod{n}$ if $x_{01},x_{10}\in C_0$. With this understood, we need 
to discuss separately according to $X=\C\P^2\# \overline{\C\P^2}$ or $X=\s^2\times \s^2$.

Suppose $X=\C\P^2\# \overline{\C\P^2}$. We write $C=(u+1)e_0-ue_1$, $C_0=(v+1)e_0-ve_1$
for some $u,v\in\Z$. Then $C_0^2+C^2=2(u+v+1)$, and $C\cdot C_0=u+v+1$. Consequently, 
in case (b), $u+v+1=b\pmod{n}$ so that $C_0^2=2(u+v+1)-C^2=2b+r^\prime\pmod{2n}$, and 
in case (c), $u+v+1=-b-r^\prime\pmod{n}$, so that $C_0^2=-2b-r^\prime \pmod{2n}$. It follows 
easily that in both cases we reached a contradiction. 

Suppose $X=\s^2\times \s^2$. We write $C=e_1+ue_2$, $C_0=e_1+ve_2$ for some $u,v\in\Z$.
Then $C_0^2+C^2=2(u+v)$ and $C\cdot C_0=u+v$. The same argument shows that 
$C_0^2=2b+r^\prime \pmod{2n}$ in case (b) and $C_0^2=-2b-r^\prime \pmod{2n}$ in case (c), 
which is a contradiction.

Note that in case (a), the above calculations also show that $C^2=C_0^2$ implies that $C,C_0$
are homologous. In this case, $\tilde{C}=C_0$, and the lemma follows.

Finally, suppose $r^\prime=0$ and case (d) occurs. Then $C_0$ gives rise to a $\s^2$-fibration
on $X$ whose fibers are embedded $J$-holomorphic two-spheres in the class of $C_0$,
which is $G$-invariant because $G$ fixes the class of $C_0$. It follows easily that
there is a $G$-invariant fiber containing the fixed-points $x_{00},x_{10}$, which is the 
desired $J$-holomorphic curve $\tilde{C}$. 

\end{proof}

We shall next construct a sequence of suitable $G$-invariant $\omega_0$-compatible
almost complex structures converging to $J_0$. To this end, let $g_0$ be the K\"{a}hler metric
and $\hat{g}_0=\omega_0(\cdot,\hat{J}_0 \cdot)$ be the metric associated to $\hat{J}_0$.
Fix a large enough $k_0>0$ and let $\rho$ be a cutoff function such that $\rho(t)\equiv 1$ for 
$t\leq 1$ and $\rho(t)\equiv 0$ for $t\geq 2$, and $0\leq \rho(t)\leq 1$ and 
$|\rho^\prime(t)|\leq 100$. Then for any integer $k>k_0$, we define a $G$-invariant metric $g_k$ 
on $X$ by 
$$
g_k(x)=g_0(x)+(\rho(k|x-x_{00}|)+\rho(k|x-x_{10}|))(\hat{g}_0(x)-g_0(x)), \;\; \forall x\in X. 
$$
Here $|x-x_{i0}|$ is the distance from $x$ to $x_{i0}$, measured with respect to the K\"{a}hler
metric $g_0$. We choose $k_0$ large enough so that (i) $|x-x_{i0}|\leq\frac{3}{k_0}$ is 
contained in the neighborhood of $x_{i0}$ where $\hat{J}_0$ is defined, and (ii) 
$\max \{k_0|x-x_{i0}|, i=0,1\}>2$. Then it follows easily that (1) $g_k$ equals $\hat{g}_0$ in a 
neighborhood of each $x_{i0}$ and converges to $g_0$ in $C^0$-topology as 
$k\rightarrow \infty$, and (2) the $C^1$-norm of $g_k$ is uniformly bounded by a constant 
depending only on $g_0$, $\omega_0$ and $\hat{J}_0$. With this understood, we let $J_k$
be the $\omega_0$-compatible almost complex structure determined by $g_k$. Then it is
clear that $J_k$ is $G$-invariant, and $J_k$ converges to $J_0$ in $C^0$-topology as 
$k\rightarrow \infty$, and the $C^1$-norm of $J_k$ is uniformly bounded by a constant 
depending only on $J_0$, $\omega_0$ and $\hat{J}_0$.

We apply Lemma 5.5 to $J_k$ and for each $k>k_0$, pick a generic $G$-invariant $J_k^\prime$ from Lemma 5.5, such that the $C^1$-norm of $J_k^\prime-J_k$ is bounded by $1/k$. Then 
$\{J_k^\prime\}$ is a sequence of $G$-invariant $\omega_0$-compatible almost complex 
structures such that
\begin{itemize}
\item $J_k^\prime$ converges to $J_0$ in $C^0$-topology as $k\rightarrow \infty$, 
\item the $C^1$-norm of $J_k^\prime$ is uniformly bounded by a constant depending only 
on $J_0$, $\omega_0$ and $\hat{J}_0$, and 
\item there is a $G$-invariant $J_k^\prime$-holomorphic $(-r^\prime)$-sphere, 
denoted by $C_k$, which contains $x_{00}$ and $x_{10}$. 
\end{itemize}

Next we shall apply the Gromov Compactness Theorem to the sequence $\{C_k\}$. To this end, 
let $\C\P^1$ be given the standard $G$-action with fixed points $0,\infty$, and let $j_0$ be the 
$G$-invariant complex structure on $\C\P^1$. Let $f_k:\C\P^1\rightarrow X$ be a $G$-equivariant 
$(J_k,j_0)$-holomorphic map which parametrizes $C_k$. Then by the Gromov Compactness Theorem, after re-parametrization if necessary, a subsequence of $f_k$ (still denoted by $f_k$ 
for simplicity) converges to a ``cusp-curve" $f_\infty: \Sigma\rightarrow X$ in the following sense 
as $k\rightarrow \infty$, where $\Sigma=\sum_\nu\Sigma_\nu$ is a nodal Riemann two-sphere.
\begin{itemize}
\item The maps $f_k$ converges to $f_\infty$ locally in H\"{o}lder $C^{1,\alpha}$-norm for some 
$\alpha>0$ (this follows from the fact that the $C^1$-norm of $J_k^\prime$ is uniformly bounded 
and by the standard elliptic estimates, e.g. cf. \cite{McDS}).
\item There is no energy loss, i.e., $(f_\infty)_\ast [\Sigma]=C$.
\item The map $f_\infty$ is $(J_0,j_0)$-holomorphic (this is due to the fact that $J_k^\prime$ converges to $J_0$ in $C^0$-topology). 
\end{itemize}
In the present situation, since each $f_k$ is $G$-equivariant, the convergence 
$f_k\rightarrow f_\infty$ is also $G$-equivariant. In particular, there is a $G$-action on the nodal Riemann two-sphere $\Sigma=\sum_\nu\Sigma_\nu$, with respect to which 
$f_\infty: \Sigma\rightarrow X$ is $G$-equivariant. It is important to note that the $G$-action on 
$\Sigma$ and the map $f_\infty$ have the following properties: 
\begin{itemize}
\item Each component $\Sigma_\nu$ is either $G$-invariant or is in a free $G$-orbit.
\item For any two distinct $G$-invariant components $\Sigma_\nu$, $\Sigma_\omega$ such that 
$z_0\in \Sigma_\nu\cap \Sigma_\omega$ which is a fixed-point of $G$, if $g_\nu,g_\omega\in G$
are the elements which act by a rotation of angle $2\pi/n$ in a neighborhood of 
$z_0\in \Sigma_\nu$ and $z_0\in \Sigma_\omega$ respectively, then $g_\omega=g_\nu^{-1}$. 
\end{itemize}
(There is an equivalent formulation of the above statements in terms of the Orbifold Gromov Compactness Theorem, see \cite{CR}.)

We proceed by finding what kind of possible components the $J_0$-holomorphic cusp-curve 
$f_\infty$ has, which are $J_0$-holomorphic two-spheres in $X$. To this end, recall that $X$ 
comes with a $J_0$-holomorphic $\C\P^1$-fibration over $\C\P^1$ which is $G$-invariant. The $J_0$-holomorphic two-spheres in $X$ are either fibers or sections of this fibration. There are 
two $G$-invariant fibers $F_0,F_1$, containing $x_{00},x_{10}$ respectively, and
two $G$-invariant sections $E_0,E_1$ which has self-intersection $-r$ and $r$ respectively, and these are the only $G$-invariant $J_0$-holomorphic two-spheres in $X$ 
(cf. Wilczynski \cite{Wil2}, \S 4).

The following observation greatly simplifies the analysis of the components of the cusp-curve 
$f_\infty$, that is,
$$
2 (f_\infty)_\ast [\Sigma]\cdot E_1=2C\cdot E_1=C^2+E_1^2=-r^\prime+r=n.
$$
(We have seen it in the proof of Lemma 5.5.) An immediate consequence of this is that there are 
no components $f_\infty(\Sigma_\nu)$ which are not $G$-invariant, because the $G$-orbit of 
such a component will contribute at least $n$ to the intersection number 
$(f_\infty)_\ast [\Sigma]\cdot E_1$, which contradicts $2 (f_\infty)_\ast [\Sigma]\cdot E_1=n$. 
Furthermore, $E_1$ can not show up in the cusp-curve either because $E_1^2=r\geq n$. 
Consequently, the only $J_0$-holomorphic two-spheres which are possibly allowed 
in the cusp-curve $f_\infty$ are the $(-r)$-section $E_0$ and the invariant fibers $F_0,F_1$. 

Next we show that if $F_i$, $i=0$ or $1$, shows up in the cusp-curve $f_\infty$, the multiplicity 
must be at least $n$, which is also not allowed by $2(f_\infty)_\ast [\Sigma]\cdot E_1=n$.  To see this, suppose 
without loss of generality that the component $f_\infty:\Sigma_\nu\rightarrow X$ has image 
$F_0$. Then the homology class $(f_\infty)_\ast [\Sigma_\nu]=m_\nu F_0$ for some $m_\nu>0$ 
such that $m_\nu=b \pmod{n}$ because the weight of the $G$-action in the direction of fiber 
$F_0$ equals $b$ at $x_{00}$. With this understood, let $z_0\in\Sigma_\nu$ such that 
$f_\infty (z_0)=x_{01}$. Then it follows easily that there must be another component 
$f_\infty:\Sigma_\omega\rightarrow X$ with $z_0\in \Sigma_\omega$. Furthermore, it must also 
have image $F_0$ because the other two allowable  $J_0$-holomorphic two-spheres $E_0$ 
and $F_1$ do not contain $x_{01}$. Let $m_\omega>0$ be the multiplicity of 
$(f_\infty)_\ast [\Sigma_\omega]$ in $F_0$. Then $m_\omega=-b\pmod{n}$, because the 
relation $g_\omega=g_\nu^{-1}$ in the Gromov Compactness Theorem we alluded to earlier, 
and the fact that the weight of the $G$-action in the direction of $F_0$ at $x_{01}$ equals $-b$. 
This proves our claim that the multiplicity of $F_0$ in the cusp-curve $f_\infty$ must be at least 
$n$ as $m_\nu+m_\omega=0\pmod{n}$. 

We conclude that $E_0$ is the only possible $J_0$-holomorphic curve in the cusp-curve 
$f_\infty$. However, this is also a contradiction because $E_0\cdot E_1=0$ but 
$(f_\infty)_\ast [\Sigma]\cdot E_1\neq 0$. This completes the proof of Proposition 5.3 when neither
$F_0$ nor $F_1$ is fixed by $G$.

\vspace{2mm}

{\bf Case (ii): Either $F_0$ or $F_1$ is fixed by $G$}. Without loss of generality, we shall only
consider the case where $F_0$ is fixed, which corresponds to $b=0$. Note that both of
$F_0,F_1$ are fixed by $G$ if and only if $b=0=r^\prime$; in particular, when $n=2$,
both $F_0,F_1$ are fixed because $b=0=r^\prime$ in this case. 

Let $x_0,x_1$ be the fixed points on $C$ such that $x_0\in F_0$. Then $x_1=x_{10}$ 
unless $F_1$ is also fixed by $G$, in which case $x_1\in F_1$. We shall fix a $\hat{J}_0$,
which is a $G$-invariant, $\omega_0$-compatible, integrable complex structure in a 
neighborhood of $x_{0}$ and $x_{1}$, such that 
\begin{itemize}
\item $\hat{J}_0$ agrees with $J_0$ at $x_{0},x_{1}$;
\item for any $i=0,1$ such that $C$ intersects $F_i$ transversely, there are holomorphic 
coordinates $z_1,z_2$ (with respect to $\hat{J}_0$) such that $C$ is 
given by $z_2=0$ and $F_i$ is given by $z_1=0$. 
\end{itemize}

Correspondingly, we have the following lemma in place of Lemma 5.5.

\begin{lemma}
For any generic $G$-invariant $J$ equaling $\hat{J}_0$ in a neighborhood of $x_{0}$ and 
$x_{1}$, there is an embedded $G$-invariant $J$-holomorphic two-sphere $\tilde{C}$ 
such that (1) $\tilde{C}$ and $C$ are homologous, and (2) $x_{10}\in \tilde{C}$.
In particular, $\tilde{C}^2=-r^\prime$. 
\end{lemma}

\begin{proof}
We apply Lemma 2.3 or 2.4 to any given $G$-invariant $J$ which equals $\hat{J}_0$ in a neighborhood of $x_{0}$ and $x_{1}$, and note that since $F_0$ is fixed by $G$,
it is automatically $J$-holomorphic, so that the $\s^2$-fibration structure on $X$ always 
exists, and we may consider $C_0$ as a $J$-holomorphic section even in case (1)
of the lemma. 

Let $y_0,y_1$ be the fixed points on $C_0$ where $y_0\in F_0$. Then $y_1=x_{10}$ or
$x_{11}$ if $r^\prime>0$,  and $y_1\in F_1$ if $r^\prime=0$. When $J$ is chosen to be
a generic $G$-invariant almost complex structure, $d\geq 0$ implies that
$$
(a) \;C_0^2=-r^\prime \mbox{ if } y_1=x_{10}, \; (b) \;C_0^2=-n+r^\prime \mbox{ if } y_1=x_{11},
\mbox{ and } (c) \; C_0^2=-n \mbox{ if } r^\prime=0. 
$$
Furthermore, case (b) or (c) can be ruled out by observing $C\cdot C_0=0 \pmod{n}$ as
we argued before in Section 4 so that
$C_0^2=r^\prime \pmod{2n}$, which contradicts the above equations. In particular, 
$r^\prime\neq 0$, so that $F_0,F_1$ can not be both fixed by $G$. The lemma follows by
taking $\tilde{C}=C_0$. 

\end{proof}

The rest of the argument is the same as in case (i), except there is one additional possibility. 
Let $J_k^\prime$ be a sequence of such generic almost complex structures converging to
$J_0$ and $C_k$ be the corresponding $J_k^\prime$-holomorphic curves from 
Lemma 5.6 which converges to a cusp-curve $f_\infty$. Denote by $x_0^{(k)}$, $x_1^{(k)}$
the fixed points on $C_k$. Then in case (i) $x_0^{(k)}=x_{00}$, $x_1^{(k)}=x_{10}$ for
all $k$, however, in the present case, $x_0^{(k)}$ can be any point on $F_0$. In particular,
there is the additional possibility that $x_0^{(k)}$ converges to $x_{01}$ as $k\rightarrow\infty$.
Let $f_\infty: \Sigma_\nu\rightarrow X$ be the corresponding component containing the
limit point $x_{01}$ of $x_0^{(k)}$. Then $(f_\infty)_\ast [\Sigma_\nu]=m_\nu F_0$ for
some $m_\nu>0$ such that $m_\nu=b \pmod{n}$. Since $b=0$ we continue to have 
$m_\nu\geq n$. As in case (i), the existence of the cusp-curve $f_\infty$ is a contradiction. 

The proof of Proposition 5.3 is completed. 

\vspace{3mm}

Proposition 5.3 has an analog for pseudo-free actions, where Lemma 5.5 requires a different 
approach as we have seen in the proof of Theorems 1.1 and 1.2. (Lemma 5.6 is irrelevant in the
case of pseudo-free actions.)
We choose to not repeat it here, and instead we refer to the corresponding results in
Wilczynski \cite{Wil2}. 

\vspace{2mm}

\noindent{\bf Proof of Theorem 1.4}

\vspace{2mm}

Suppose $F_r(a,b)$ and $F_{r^\prime}(a^\prime,b^\prime)$ are orientation-preservingly 
equivariantly diffeomorphic. By Lemmas 5.1 and 5.2, after modifying $F_r(a,b)$ and 
$F_{r^\prime}(a^\prime,b^\prime)$ by a sequence of canonical equivariant diffeomorphisms,
we have either $F_r(a,b)=F_{r^\prime}(a^\prime,b^\prime)$, in which case Theorem 1.4
follows, or $\gcd (a,n)=\gcd (a^\prime,n)=1$ and one of the following occurs:
\begin{itemize}
\item $a^\prime=a$, $b^\prime=b$, and $r^\prime=r-n$, or
\item $a^\prime=b$, $b^\prime=a$ and $r^\prime=r=n$, where $a\neq \pm b$. (Note that the
$G$-actions are pseudo-free in this case.)
\end{itemize}

For pseudo-free actions and when $n>2$, the above possibilities are ruled out by
Lemma 4.9(4),(5) in Wilczynski \cite{Wil2}. For non-pseudo-free actions where $n>2$
or $n=2$ and $r,r^\prime$ are even, Proposition 5.3 can be used.

It remains to examine the case where $n=2$ and either the $G$-actions are pseudo-free or
$r,r^\prime$ are odd. Observe that in this case, with an application of $c_2$ if necessary,
one can always arrange so that $b=b^\prime=1$. Thus as in the proof of Proposition 5.3, 
it suffices to examine 
$F_r(1,1)$ and $F_{r^\prime}(1,1)$ for $(r^\prime,r)=(0,2)$ or $(1,3)$. In the former case,
the two $G$-Hirzebruch surfaces are equivariantly diffeomorphic by $c_6$, and in the
latter case, by $c_3\circ c_6$. This finishes off the proof of Theorem 1.4.

\section{Symplectic $G$-minimality versus smooth $G$-minimality}

We begin by considering minimal rational $G$-surfaces $X$ where $G$ is not necessarily 
a finite cyclic group. Our first lemma shows that we only need to look at the cases where $X$ is a 
conic bundle with singular fibers or a Hirzebruch surface. 

\begin{lemma}
If $X$ is a minimal rational $G$-surface which is not minimal as a topological
$G$-manifold, then $X$ must be a conic bundle with singular fibers or a Hirzebruch surface
$F_r$ with $r>1$ and odd. 
\end{lemma}

\begin{proof}

Our first observation is that the dimension of $H^2(X;\R)^G$ is at least $2$. To see this, note that 
$X$ has a $G$-invariant K\"{a}hler form $\omega_0$ with $[\omega_0]\in H^2(X;\R)^G$ and 
$[\omega_0]^2>0$, and on the other hand, the class $E$ of the union of $(-1)$-spheres along 
which $X$ is blown down also lies in $H^2(X;\R)^G$ and $E^2<0$. It follows that $[\omega_0]$ 
and $E$ are linearly independent, and consequently, the dimension of $H^2(X;\R)^G$ is at least 
$2$. With this understood, we recall the fact that for any minimal rational $G$-surfaces $X$, 
$Pic(X)^G$ has rank at most $2$, cf. \cite{DI}, Theorem 3.8. It follows that $Pic(X)^G$ must 
have rank $2$. By Theorem 3.8 in \cite{DI} again, $X$ must be either a conic bundle with singular
fibers, or a Hirzebruch surface. Finally, we note that if $X$ can be blown down $G$-equivariantly,  
the underlying manifold $X$ must be topologically non-minimal. The lemma follows easily. 

\end{proof}

Now we specialize to the case where $G=\Z_n$ is a finite cyclic group. First, consider the case when
$X$ is a conic bundle with singular fibers. An element $\mu\in G$ which is a generator is called a
de Jonqui\'{e}res element, and such elements have been classified by Blanc \cite{B}. In particular,
$n=2m$ must be even, and $\tau=\mu^m$ is a de Jonqui\'{e}res involution, which has the following 
properties: The involution $\tau$ leaves each fiber of the conic bundle invariant, switches the two 
irreducible components in each singular fiber, and the fixed-point set of $\tau$ is an irreducible 
smooth bisection $\Sigma$ with a hyperelliptic involution, such that the conic bundle projection 
defines the quotient map with ramification points equal to the singular points of the fibers. The de Jonqui\'{e}res element $\mu$ induces a permutation of the set of singular fibers. It follows easily 
that $X$ as a $\langle \tau\rangle$-surface is also minimal.

\begin{proposition}
Let $X$ be a minimal $G$-conic bundle with singular fibers where $G=\Z_n$. 
Then $X$ is minimal as a topological $G$-manifold. 
\end{proposition} 

\begin{proof}

By the description of de Jonqui\'{e}res elements in \cite{B}, it suffices to consider only the case 
$G=\Z_2$ generated by a de Jonqui\'{e}res involution $\tau$. Let $\Sigma$ be the fixed-point set
of $\tau$, and let $k$ be the number of singular fibers. Then $k=2+2g$ where $g$ is the genus 
of $\Sigma$. Note that $k\geq 4$ (cf. \cite{DI}, Lemma 5.1), which implies that $g\geq 1$. Let 
$F$ denote the fiber class of the conic bundle. Then $F$ and $\Sigma$ span $H^2(X;\Q)^G$ 
over $\Q$ as $F\in Pic(X)^G$. We determine in the next lemma the intersection form of $F,\Sigma$.

\begin{lemma}
$F^2=0$, $\Sigma\cdot F=2$, and $\Sigma\cdot \Sigma=2+2g$, where $g\geq 1$ is the 
genus of $\Sigma$. 
\end{lemma}

\begin{proof}

It is clear that $F^2=0$ and $\Sigma\cdot F=2$ (note that $\Sigma$ is a bisection). We shall 
prove that $\Sigma \cdot \Sigma=2+2g$. To see this, let $K_X$ be the canonical class of the 
rational surface $X$. Then $K_X\in Pic(X)^G$ implies that $K_X$ is a linear combination of 
$\Sigma$ and $F$. With $K_X \cdot F=-2$ (by the adjunction formula), $\Sigma \cdot F=2$ 
and $F^2=0$, it follows easily that $K_X=-\Sigma +l\cdot F$ for some $l\in \Z$. 

Now the adjunction formula $K_X \cdot \Sigma +\Sigma\cdot \Sigma +2=2g$
gives $2l+2=2g$, which implies that $l=g-1$. 

To compute $\Sigma \cdot \Sigma$, note that
$$
\Sigma\cdot \Sigma -4l =K_X^2=8-k=8-(2+2g)=6-2g,
$$
which gives $\Sigma\cdot \Sigma = 6-2g+4l=6-2g+4(g-1)=2+2g$.

\end{proof}

Secondly, we show that $X$ contains no $G$-invariant locally linear $(-1)$-spheres. 
Suppose to the contrary that there is such a $(-1)$-sphere $C$. Obviously
the action of $G$ on $C$ is effective and orientation-preserving, so that $C$ contains
exactly two fixed points of $G$. Let $m_1,m_2$ be the weights of the $G$-action in the
normal direction of $C$ at the two fixed points. Then $m_1,m_2$ obey the following
equation: $m_1+m_2\equiv -1 \pmod{2}$. This implies that exactly one of $m_1,m_2$
equals $0$ mod $2$, and consequently, the $G$-action has an isolated fixed-point, which is
a contradiction. Hence the claim.

As a corollary, if $X$ is not minimal as a topological $G$-manifold, there must be two disjoint 
$(-1)$-spheres which are disjoint from the fixed-point set $\Sigma$, such that the $\Z_2$-action permutes the two $(-1)$-spheres. Denote their classes by $E_1$ and $E_2$ and let 
$E=E_1+E_2$. Then $E$ satisfies the following conditions:
$$
E\in H^2(X;\Z)^G, E\cdot \Sigma=0, E^2=-2, \mbox{ and } F\cdot E=0 \pmod{2}.
$$
To see the last condition, let $\tau\in G$ be the involution. Then 
$$
F\cdot E=F\cdot (E_1+E_2)=F\cdot (E_1+\tau(E_1))=F\cdot E_1+\tau(F)\cdot E_1=2F\cdot E_1
=0\pmod{2}.
$$
We write $E=a\Sigma +b F$ for some $a,b\in\Q$. Then $E\cdot \Sigma=0$
gives $a\Sigma\cdot\Sigma+bF\cdot \Sigma=0$, which, with $\Sigma\cdot\Sigma=2+2g$,
implies $b=-a(1+g)$. Now $E^2=-2$ means 
$$
(2+2g)a^2+4ab=-2,
$$
which gives $a^2(1+g)=1$. Finally, with $F\cdot E=0 \pmod{2}$, we get $2a=F\cdot E=2k$ for
some $k\in\Z$, so that $a\in \Z$. But this contradicts $a^2(1+g)=1$ as $g\geq 1$. The proof of
Proposition 6.2 is completed. 

\end{proof}

It is clear that Theorem 1.5 follows from Proposition 6.2 and Theorem 1.2. 

\vspace{3mm}

We end this section with a proof of Theorem 1.0.

\vspace{3mm}

\noindent{\bf Proof of Theorem 1.0}

\vspace{3mm}

We shall first consider the case where $X$ is symplectic. Let $\omega$ be a symplectic form on $X$ which is $G$-invariant. We shall prove that if $X$ is not minimal as a smooth four-manifold, then for any $G$-invariant, $\omega$-compatible almost complex structure $J$, there is a $G$-invariant set of disjoint union of $J$-holomorphic $(-1)$-spheres in $X$. The theorem follows easily from this when $X$ is either complex K\"{a}hler or symplectic.

According to Lemma 2.3 in \cite{C4}, since $X$ is neither rational nor ruled, $(X, \omega)$ has the 
following properties: for any $\omega$-compatible almost complex structure $J$, 
(i) $X$ contains a $J$-holomorphic $(-1)$-sphere provided that $X$ is not minimal ,
(ii) either $c_1(K_X)$ or $c_1(2K_X)$ is represented by a finite set of $J$-holomorphic curves
(for (ii) see the proof of Lemma 2.3 in \cite{C4}). With this understood, we now assume that $J$ is 
$G$-invariant. Let $C$ be a $J$-holomorphic $(-1)$-sphere in $X$, which exists by (i). Then for 
any $g\in G$, $g(C)$ is also a $J$-holomorphic $(-1)$-sphere. The proposition follows 
easily in this case if for any $g\in G$, either $g(C)=C$ or $g(C)$ is disjoint from $C$. Suppose to the contrary, there is a $g\in G$ such that $g(C)$ and $C$ are distinct with nonempty intersection.  Then by a standard gluing argument (cf. \cite{S}), one can produce an irreducible $J$-holomorphic 
curve $C^\prime$ such that $C^\prime=C+g(C)$. Now observe that the self-intersection 
$$
C^\prime\cdot C^\prime= C^2+2 C\cdot g(C) +g(C)^2\geq -1+2\cdot 1+(-1)=0.
$$
On the other hand, $c_1(K_X)\cdot C^\prime= c_1(K_X) \cdot (C+g(C))=-2$. But this clearly contradicts the fact that either $c_1(K_X)$ or $c_1(2K_X)$ is represented by a finite set of 
$J$-holomorphic curves because $C^\prime\cdot C^\prime\geq 0$. 

It remains to consider the case where $X$ is a non-K\"{a}hler complex surface. In this case, 
since $b_2^{+}\geq 1$, $X$ must be an elliptic surface, and if $X$ is minimal as a complex 
surface, the elliptic fibration on $X$ must be either nonsingular or have at most multiple fibers 
(cf. \cite{FM}, Theorem 7.7). Furthermore, the condition $b_2^{+}\geq 1$ implies that the base 
curve of the elliptic fibration has non-zero genus. It follows easily that $X$ is a minimal complex surface if and only if $X$ is minimal as a smooth four-manifold. Consequently, if $X$ is not 
minimal as a smooth four-manifold, the elliptic fibration must be non-minimal and all the 
exceptional curves in $X$ are contained in the fibers. Observe that any two exceptional curves 
in the same fiber must be disjoint. This implies that the exceptional curves in $X$ are all disjoint. 
Any $G$-orbit of an exceptional curve in $X$ gives rise to a $G$-invariant set of disjoint union of exceptional curves, hence $X$ is non-minimal as a complex $G$-surface. This finishes off the proof of the theorem.

\vspace{2mm}

{\Small University of Massachusetts, Amherst.\\
{\it E-mail:} wchen@math.umass.edu

\end{document}